\documentclass[12pt]{article}
 \expandafter\let\csname equation*\endcsname\relax
 \expandafter\let\csname endequation*\endcsname\relax
\usepackage{bm}
\usepackage{enumerate}
\usepackage[top=1in, bottom=1in, left=1in, right=1in]{geometry}
\usepackage{mathrsfs,color}
\usepackage{amsfonts}\usepackage{multirow}
\usepackage{amssymb}
\usepackage{caption,comment}
\usepackage{blkarray}
\usepackage{dsfont}
\usepackage{todonotes}
\usepackage{amsmath}
\usepackage{wrapfig}
\usepackage{float}
\usepackage{amssymb,amsthm}
\usepackage[toc,page]{appendix}
\usepackage{graphicx,afterpage}
\usepackage{epstopdf}

\usepackage[sort,compress]{cite}

\usepackage{xcolor}
\usepackage{tikz}
\numberwithin{equation}{section}
\numberwithin{figure}{section}
\def\theequation{\arabic{section}.\arabic{equation}}

\newcommand\CE{\mathcal{E}}
\newcommand\reals{\mathbb{R}}
\newcommand\E{\mathbb{E}}
\newcommand\Prob{\mathbb{P}}
\newcommand\rmd{\mathrm{d}}
\newcommand\rmi{\mathrm{i}}
\newcommand\btilde{\tilde{b}}

\def\thefigure{\arabic{section}.\arabic{figure}}
\newtheorem{thm}{Theorem}[section]
\newtheorem{cor}[thm]{Corollary}
\newtheorem{lem}[thm]{Lemma}
\newtheorem{prop}[thm]{Proposition}
\newtheorem{defn}[thm]{Definition}

\newtheorem{aspt}[thm]{Assumption}

\begin{document}
\title{Simple Nonlinear Models with Rigorous Extreme Events and Heavy Tails }
\author{Andrew J  Majda and  Xin T Tong}
\maketitle
\begin{abstract}
Extreme events and the heavy tail distributions driven by them are  ubiquitous in various scientific, engineering and financial research. 
They are typically associated with stochastic instability caused by hidden unresolved processes.
Previous studies have shown that such instability can be modeled by a stochastic damping in conditional Gaussian models.
However, these results are mostly obtained through numerical experiments,  
while a rigorous understanding of the underlying mechanism is sorely lacking. 
This paper contributes to this issue by establishing a theoretical framework,
in which  the tail density of conditional Gaussian models can be rigorously determined.
In rough words, we show that if the stochastic damping takes negative values, the tail is polynomial;
if the stochastic damping is nonnegative but takes value zero at a point,  the tail is between exponential and Gaussian.
The proof is established by constructing a novel, product-type Lyapunov function, where a Feynman-Kac formula  is applied. 
The same framework also leads to a non-asymptotic large deviation bound for long-time averaging processes.
\end{abstract}
\section{Introduction}
With dramatic global climate change in recent years, extreme climate events, along with their destructive power, are observed more often than ever. Severe heatwaves reduce crop harvest, increase forest fire risk and sometimes lead to human casualties. Heavy downpours, as another extreme, can flood large areas and cause significant economic losses \cite{Easterling_etal}.  
Extreme events are also of great interest in engineering and financial research, because of the underlying risk. 
Rogue waves, seen as walls of waters of 10 meters high, can easily sink unprepared ships \cite{FS17}. 
Credit default of one bank can  lead to world-wide financial recession \cite{EMN12}. 
The capability to model, measure and predict these extreme events has never been so important \cite{HD14, CGHM14, tsay05}.  

Mathematically, extreme events can be viewed as strong anomalies seen in the time series of certain observables.
Collectively, they produce an exponential or even polynomial heavy tail in the observable's histogram. 
They often appear in complex nonlinear models that have stochastic instability.  
This instability typically comes as a combined effect of many hidden or unresolved processes \cite{CGHM14}.  
Examples of these hidden processes include cloud formation, precipitation and refined scale turbulence \cite{MFK08, MS09, KBM10, KMS13}.
For these processes, only limited direct observations are available.  
Accurate physical models of them are lacking, or require expensive computation.
A better modeling strategy is viewing them as stochastic processes, of which the parameters can be tuned to fit data statistically \cite{Majda_Wang}. 

The stochastic instability discussed above can be described by the following simple nonlinear model:
\begin{equation}
\label{sys:dyd}
\begin{gathered}
\rmd X_t=-b(u_t) X_t \rmd t+ \sigma_x \rmd W_t,\\
\rmd u_t=h(u_t)\rmd t+\rmd B_t.
\end{gathered}
\end{equation}
Here,   $X_t$ represents certain observables of a physical model, 
while its dynamics is affected by a hidden process $u_t$. Let $d_X$ and $d_u$ be the dimension of $X_t$ and $u_t$.
$W_t$ and $B_t$ are two independent Wiener processes of dimensions $d_X$ and $d_u$ respectively. 
For simplicity, we assume throughout the paper that $u_t$ is ergodic, 
and $\pi$ is its equilibrium distribution.

In \eqref{sys:dyd}, the stochastic instability is represented by the damping rate $b(u_t)\in\reals$. 
The dynamics of $X_t$ is unstable if  $b(u_t)$ is zero or negative in an interval of time, 
strong large spikes will appear in the trajectory of $X_t$ as a consequence, 
which we can interpret as extreme events. Note that this is a random event that takes place intermittently,
since it is triggered by the random realization of the process $u_t$.

One important feature of model \eqref{sys:dyd} is that the dynamics of $X_t$ is linear if $u_t$ is fixed. 
We can generalize the formulation in \eqref{sys:dyd} and maintain this feature. Consider
\begin{equation}
\label{sys:condGauss}
\begin{gathered}
\rmd X_t=-B(u_t)X_t  \rmd t +\Sigma_X \rmd W_t,\\
\rmd u_t=h(u_t)\rmd t+\rmd B_t,
\end{gathered}
\end{equation}
where $B(u)$ is a matrix valued function. This is known as a conditional Gaussian system \cite{LS01}. 
This formulation can be found in many nonlinear models, such as stochastic parameterization Kalman filter (SPEKF), Lagrangian floater, low order Madden Julian Oscillation model, and turbulent tracers \cite{CMT14, CMT14b, CMG14, CM18, CMT17, MT15}.  The conditional Gaussian structure can be exploited for efficient  computations. If we apply the vanilla Monte Carlo method to estimate the density of  $X_t$,  the necessary sample size is $e^{d_X}$, which is prohibitive when $d_X$ is large. But knowing that $X_t$ is conditionally Gaussian, it suffices to compute the conditional mean and covariance, of which the computational cost only scale cubically with $d_X$. This feature can be exploited for high dimensional  prediction and data assimilation \cite{CMG14, CMT14, CMT14b, CM18, CMT17}. 

Conditional Gaussian model is known to be a good tool for studying extreme events and heavy tail phenomena. 
In the  SPEKF model, an observable $x_t$ is driven by
\[
\rmd x_t= (-\gamma(t) +\rmi \omega) x_t  \rmd t+f(t) \rmd t+\sigma_x \rmd W_t,
\]
where $\gamma(t)$ and $f(t)$ are independent Ornstein-Uhlenbeck (OU) processes modeling the unobservable instability and forcing, and `$\rmi$' is the imaginary unit. This simple nonlinear model was first introduced in \cite{GHM10} for filtering multiscale turbulent signals with hidden instabilities,  and later used for filtering, prediction, parameter estimation in the presence of model error \cite{GM11, BGM12, Majda_Harlim_Gershgorin, CGHM14}. Another example is the turbulent passive tracer. Passive tracers are substance transported by a turbulence. 
They can reveal many important properties of the underlying turbulence and have important environmental impacts \cite{Neelin_etal, MG13, LMQ16}. Mathematically, given a turbulence velocity field $V$, the passive tracer density $T(x,t)$ follows an advection-diffusion equation:
\[
\partial_t T+V\cdot \nabla T=-\gamma_T T+\kappa \Delta T.
\]
This dynamic is linear conditioned on $V$, so $T$ can be interpreted as $X_t$ in \eqref{sys:condGauss}. 
And in order to have a good numerical representation of the 2-dimensional density field, $T$ needs to include hundreds or thousands of Fourier modes. In other words, $X_t$ is high dimension, and the aforementioned conditional Gaussian computational strategy is more efficient than  vanilla Monte Carlo. 
Former numerical experiments indicate that even with a simple zonal sweep  $V$, 
the passive tracers have extreme events and an exponential-like histogram. 
This is in accordance with the laboratory observations such as the classical Rayleigh-Bernard convection \cite{Cetal89, Gollub_etal} and readings from the atmosphere \cite{Neelin_etal}.
An earlier result of the authors \cite{MT15} has rigorously explained this phenomenon using a delicate phase resonance. 

Despite the extensive success in using conditional Gaussian models for extreme event research,
most findings are justified by numerical experiments. 
The only rigorous result \cite{MT15} focuses only on a specific passive tracer model.
There lacks a rigorous extreme event framework that applies to the general model \eqref{sys:dyd} or \eqref{sys:condGauss}.
This can be problematic, since extreme events are typically rare and can be very difficult to simulate or observe. Experimental data, therefore, can be inaccurate,  especially if the model contains many variables or has complicated nonlinearity. 

This paper intends to close this gap by giving concrete criteria that lead to provable heavy tails of $X_t$. As a result, when a model of type \eqref{sys:dyd} or \eqref{sys:condGauss} is available, we know apriori the tail density of  $\|X_t\|$. This will be extremely helpful to the stochastic modeling of extreme events, as it turns a nonparametric problem parametric. Furthermore, we can obtain lower and upper bounds of the shape parameters of these distributions, and in some simple cases, these bounds are sharp.  From the reverse perspective, when we only have data of $X_t$ and intend to fit it with a model, the criteria in this paper can be used as guidelines for the choice of the model.


\subsection{Main results in a simplified setting}
In order to give a quick idea of  our main result, consider an unforced SPEKF model with general damping \cite{CGHM14}:
 \begin{equation}
\label{eqn:2OU}
\begin{gathered}
\rmd X_t=-b(u_t) X_t \rmd t+ \rmd W_t,\\
\rmd u_t=-\gamma u_t \rmd t+ \rmd B_t.
\end{gathered}
\end{equation}
We assume $u_0$ follows the equilibrium distribution $\pi=\mathcal{N}(0,\tfrac{1}{2\gamma})$ and $X_0=0$. Our result indicates that the tail of $X_t$ is controlled by some simple properties of the damping function $b$:
\begin{thm}
\label{thm:simple}
In model \eqref{eqn:2OU}, suppose the damping on average is positive, that is
\[
\langle \pi, b\rangle:=\int  b(u) \pi(\rmd u)=\E b(u_t)>0,
\] and $b$ is Lipschitz, then $(X_t, u_t)\in \reals^{1+1}$  has a unique equilibrium distribution, under which the tail of $|X_t|$ is
\begin{enumerate}[i)]
\item polynomial, if $b$ can take negative values.
\item exponential,  if $b(u)$ is nonnegative and takes value $0$ in an interval.
\item between exponential and Gaussian, if $b$ is nonnegative, and takes value $0$ at a point. 
\item Gaussian,  if $b$ is bounded away from zero. 
\end{enumerate}
\end{thm}
The difference between case ii) and iii) is quite subtle but important: in case ii), $b(u_t)$ takes value zero with positive amount of time, while in case iii), $b(u_t)$ can be close to zero, but takes value zero only at a singular set of times. We also emphasize that this simplified setting can be  generalized and include systems \eqref{sys:dyd} and \eqref{sys:condGauss}. The Lipschitz requirement of $b$ can be relaxed to include all the test functions discussed directly below. The general statements can be found in Theorems \ref{thm:unstable}, \ref{thm:exponential} and \ref{thm:Gaussiantail}. Proof is allocated in \ref{sec:appenduni}. 

As a quick verification of  Theorem \ref{thm:simple}, we conduct several simple numerical experiments. For each experiment, model \eqref{eqn:2OU} with $\gamma=2$ is simulated for an extensive length $T=10^6$. An implicit Euler scheme \cite{HMS06} is implemented for the $X_t$ part with a small time step $\Delta t=10^{-2}$, so the large anomalies come not as a result of numerical instability. The realization of the unobservable process $u_t$ is kept the same for comparison. We present the  trajectory of $|X_t|$ for $t\in [9,000,10,000]$ to demonstrate the extreme events. We also present the log-density plot based on the histogram of $T/\delta t=10^8$ data points. A Gaussian density with the same mean and variance is plotted as a reference.

In the first group of experiments, we consider damping functions that can take negative values. Following the example of unforced SPEKF model in \cite{CGHM14}, we test affine functions $b(u)=1+cu$, where $c=3,2,1$. The intercept $1$  is necessary for $\E b(u_t)>0$. The results are presented in Figure \ref{fig:unstable}. We can see clearly that the trajectories of $|X_t|$ are filled with strong intermittent extremal anomalies. And with the increment of $c$, the amplitudes of the extreme events grow exponentially. This can also be seen in the log-density plots, which all have a logarithmic tail profile, while the range increases with $c$. This indicates the tails are indeed polynomial like.  In particular, when $c=3$, the theoretical variance of $X_t$ is infinite, which we will find out in Section \ref{sec:bivariate}. The sample variance exceeds $10^5$ because of the extremal anomalies. We do not plot the Gaussian density reference as it is invalid.  

\begin{figure}
 \hspace*{-2cm}
\includegraphics[scale=0.4]{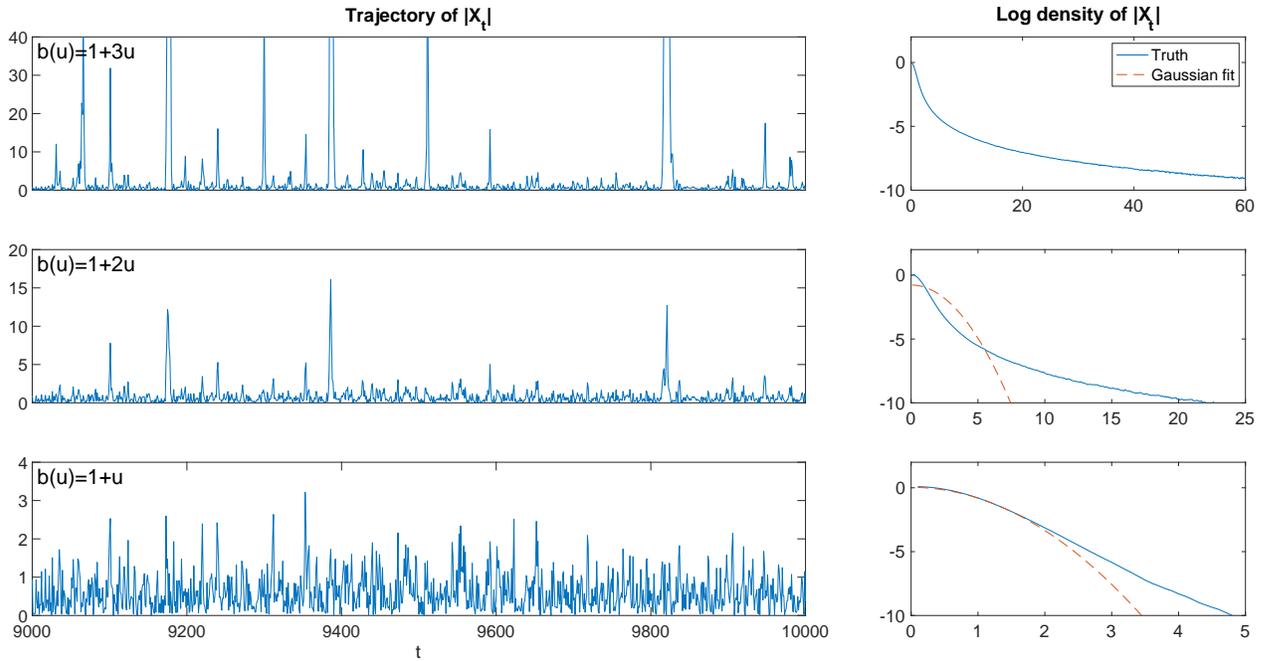}
\caption{Unstable dampings lead to polynomial tails. The damping function being used is labeled at the top left corner of each panel. For the $b(u)=1+3u$ case, the Gaussian fit is invalid since the theoretical variance is infinite.}
\label{fig:unstable}
\end{figure}

In the second group of experiments, we consider damping functions that are nonnegative, but take value $0$ in intervals. We test with piecewise linear functions 
\[
b(u_t)=2[|u_t+1|-1]^+,\quad b(u_t)=[|u_t+1|-1]^+,\quad\text{and } b(u_t)=[|u_t+1|-0.5]^+.
\] 
Here $[x]^+:=\max\{x,0\}$ takes the positive part of the input.  The results are presented in Figure \ref{fig:exp}. We can see that the trajectories of $|X_t|$ are filled with extreme events of various types. The log-density plots all have  linear profiles, which indicates that the tails are exponential.
\begin{figure}
 \hspace*{-2cm}
\includegraphics[scale=0.4]{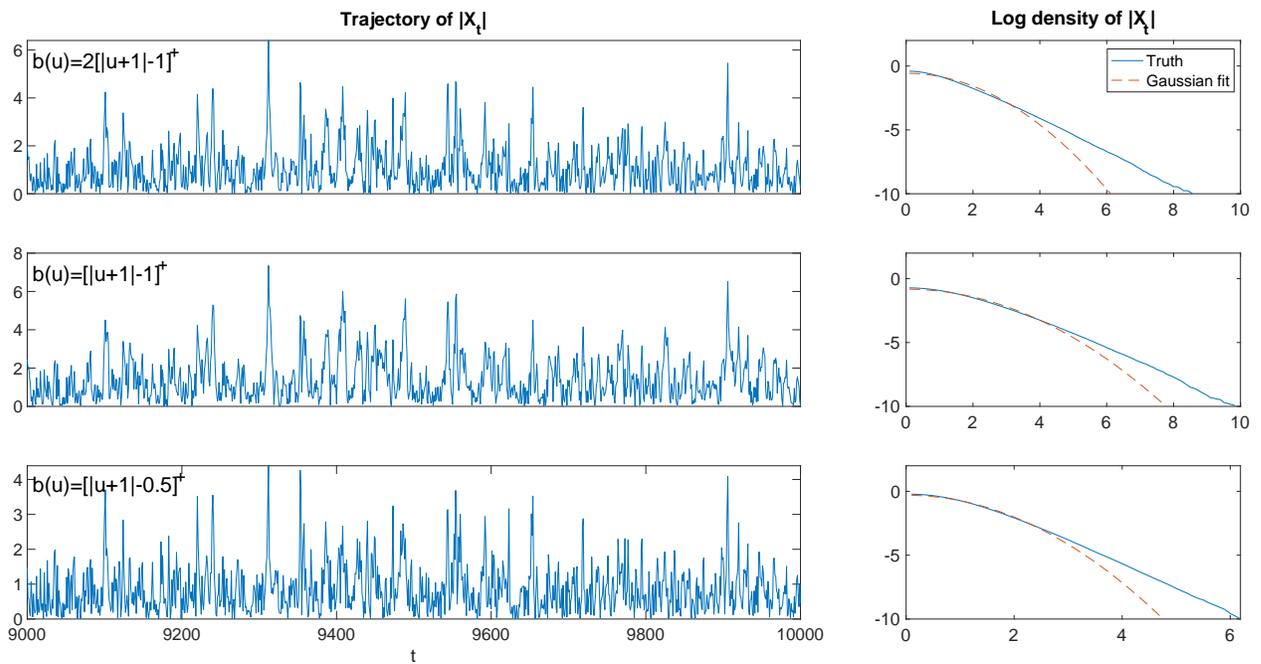}
\caption{Nonnegative dampings that take value  zero near $u=-1$ lead to exponential tails. The damping function being used is labeled at the top left corner of each panel. }
\label{fig:exp}
\end{figure}

In the third group of experiments, we consider damping functions that are nonnegative, but take value $0$ only at the origin. We test with functions $b(u)=|u|^c$, where $c=4,2,1$. The results are presented in Figure \ref{fig:expgauss}. From both the trajectory plots and the density plots, we find that with a larger $c$, the anomalies last longer, and the tails are more like exponential. And for $c=1$, the plots are quite similar to the OU case studied next. 

\begin{figure}
 \hspace*{-2cm}
\includegraphics[scale=0.4]{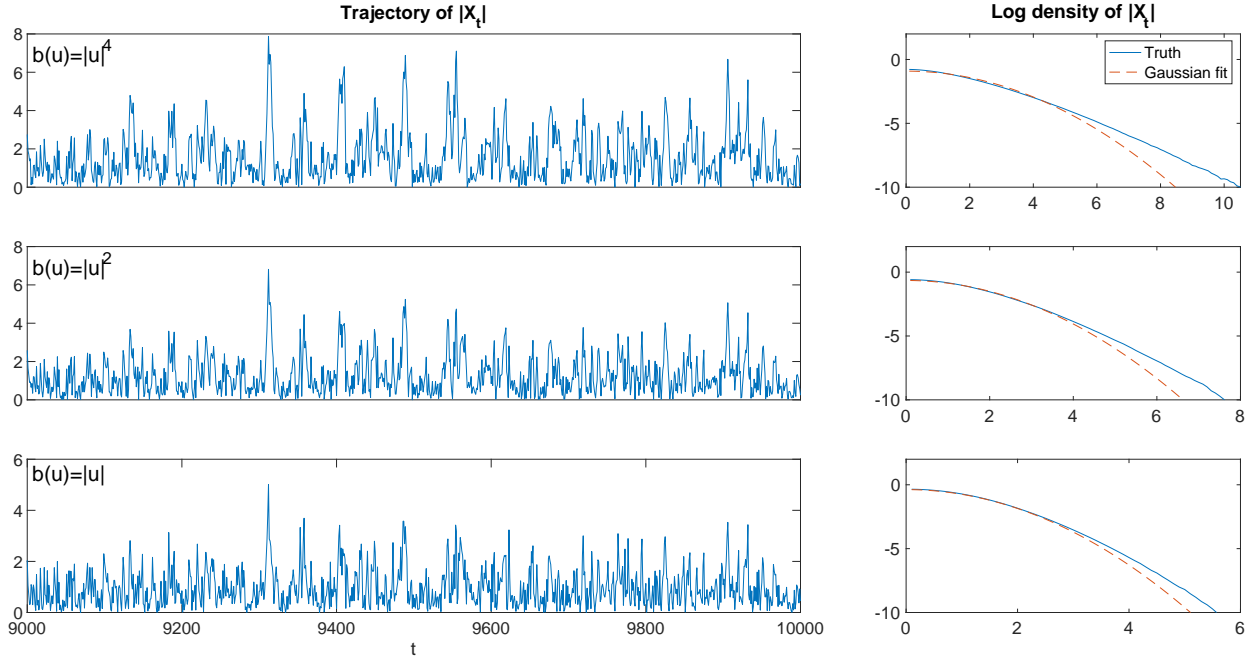}
\caption{ Nonnegative dampings that take value zero at the origin lead to tails between exponential and Gaussian. The damping function being used is labeled at the top left corner of each panel. }
\label{fig:expgauss}
\end{figure}

In the final experiments, we consider damping functions that are strictly positive
\[
b(u_t)=|u_t|^4+1,\quad b(u_t)=|u_t|^2+1,\quad\text{and } b(u_t)\equiv 1.
\] 
So in the last experiment, $X_t$ is simply OU.  The results are presented in Figure \ref{fig:Gauss}. We see that the densities are fitted very well with Gaussian approximations. Moreover, the trajectories are all very similar. 

\begin{figure}
 \hspace*{-2cm}
\includegraphics[scale=0.4]{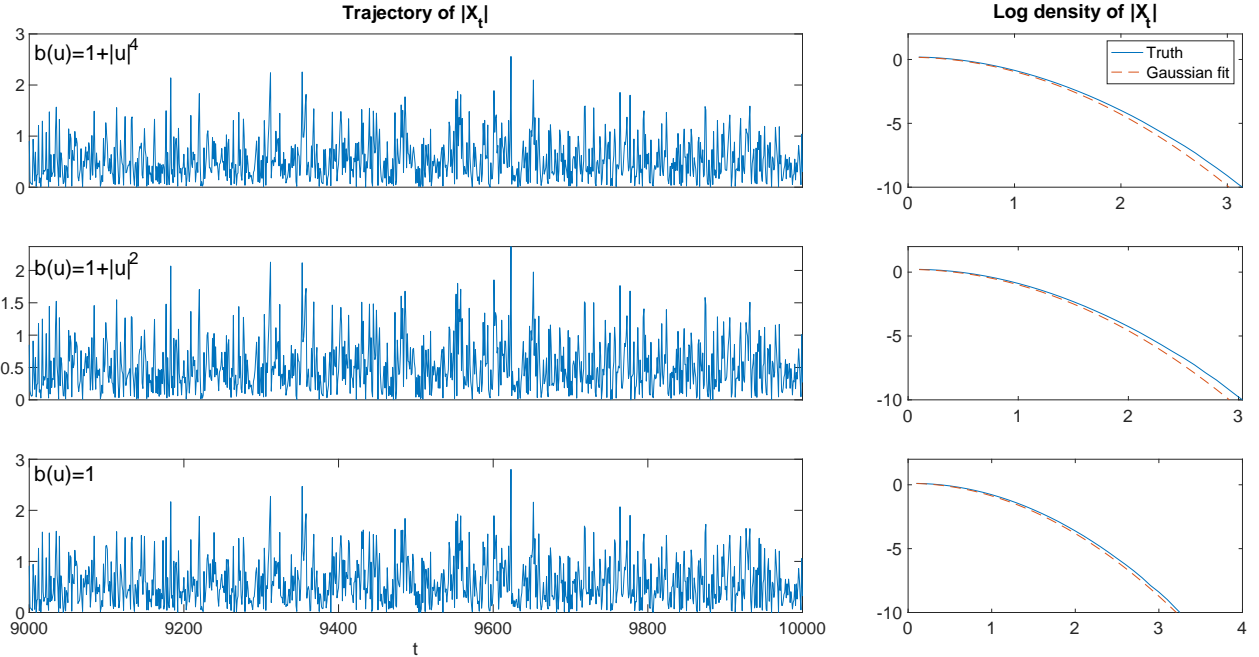}
\caption{Strictly positive dampings lead to Gaussian tails. The damping function being used is labeled at the top left corner of each panel. For the $b(u)=1$ case, $X_t$ is simply an OU process.}
\label{fig:Gauss}
\end{figure}

As a quick summary, the simulation results are in accordance with the predictions made in Theorem \ref{thm:simple}. We  can also see that some simple changes in the damping function can lead to vastly different types of intermittency and heavy tail distributions. In practice, Figure \ref{fig:unstable}-\ref{fig:Gauss} can be used as  references for modeling extreme events.

\subsection{Moment behaviors}
To determine the tail type of $X_t$, we will consider the moments $\mathbb{E}\|X_t\|^p$ of different power $p$. By investigating the moments of density functions like $c x^{-p}, \exp(-c x)$ and $\exp(-c x^2)$,  we see the moments of a random variable $X$ have very different behavior, depending on the distribution of $X$:
\begin{itemize}
\item Polynomial like: $\E X^p<\infty$ if and only if $p$ is below a threshold.
\item Exponential like:  $\E X^p<\infty$ for all $p>0$ and $\log\E X^{2p}\propto 2p\log p$ for large $p$.
\item Gaussian like: $\E X^p<\infty$ and $\log\E X^{2p}\propto p\log p$ for large $p$.  
\end{itemize}
Such difference can be used to obtain the classification in Theorem \ref{thm:simple}.  As we will see, the high moments are very sensitive to the behavior of the damping function $b$. 

Similar moment behaviours can be found in other stochastic models as well. 
Another way to model stochastic intermittency, is to model $X_t$ as in \eqref{sys:dyd}, and let $u_t$ be a continuous time Markov jump process \cite{MH12}. 
Such a model is known as a Markov switching or regime switching diffusion \cite{BGM10, CH14, Sha15}. It is used in atmospheric science to model complex cloud precipitation, in filtering theory to represent model error, and in financial time series to model hidden market behavior \cite{MS09, MS09b,MH12, tsay05}. 

Markov switching models can also produce heavy tail distributions. In fact, the quadchotomy in Theorem \ref{thm:simple} has a similar version for finite state Markov switching models in \cite{BGM10}. There are further efforts to generalize this result to infinite state spaces \cite{CH14, Sha15}, and to investigate the regularity of underlying measures \cite{BH12, BHM15}. Yet these results often require a life-death process in the background, which limits their range of application.

While the theoretical result here can be interpreted as an extension of \cite{BGM10},  such extension is nontrivial. Stochastic differential equations (SDE) are natural tools when modeling physical processes of continuous values. Approximating them as Markov jump processes on a meshgrid is often inappropriate, because the number of grid points scales exponentially with the dimension $d_X+d_u$, and becomes too large for physics related problems, for example the passive tracer field model mentioned earlier. Many physical concepts such as energy dissipation and flow contraction are usually understood only through SDE formulation. The new results in below will reveal the important connection between these physical concepts and the heavy tail phenomena. 

Moreover, this paper employs a different analysis framework from the ones used in \cite{BGM10, CH14}. In the previous framework, the moment analysis is established by investigating the amount of time the Markov jump process spends in each state. This is difficult to be generalized for an SDE. In this paper, the estimates are constructed by finding novel product type Lyapunov functions. A similar strategy can also be implemented on Markov switching processes to understand complicated geometric ergodicity and multi-scale behaviors \cite{MT16RMS}. 

Apart from SDE, moment analysis can also be conducted for stochastic partial differential equations (SPDE) \cite{CD15, CKK17, GX17}. This has been applied to understand the regularity, growth speed, and localization of the SPDE solutions. So far, these results apply to specific SPDEs, for example the heat equation and the Anderson model. In comparison, our requirements for the SDE are rather general. It will be interesting if the analysis framework developed here can be applied to SPDEs as well.

\subsection{Connection to large deviations of trajectory average}
Interestingly, our result also leads to a non-asymptotic large deviation bound for trajectory average. Given any function $b$, by the Birkhoff ergodic theorem, we know
\[
t^{-1}\int^t_0 b(u_s)\rmd s\overset{t\to \infty}{\longrightarrow} \langle \pi, b\rangle:=\int b(u) \pi(du). 
\]
Such convergence has been used routinely to compute $ \langle \pi, b\rangle$, known as the Markov Chain Monte Carlo method. It is natural to ask how does the deviation 
\[
D_t:=t^{-1}\int^t_0 b(u_s)\rmd s-\langle \pi, b\rangle
\] 
converge to zero as $t$ becomes large. 

To see how our study of system \eqref{sys:dyd} connects with this problem, we let 
\begin{equation}
\label{eqn:XtLDP}
X_t:=\exp\left(\int^t_0 (b(u_s)-\langle \pi, b\rangle) \rmd s\right).
\end{equation}
Then clearly $X_t>0$ follows the ordinary differential equation (ODE) $\dot{X}_t=(b(u_t)-\langle \pi,b\rangle) X_t$, and it fits in the formulation \eqref{sys:dyd} with $\sigma_x=0$. A large deviation bound of $D_t$ can be obtained by finding the moments of $X_t$ and then apply the Markov inequality,
\[
\Prob(D_t \geq c) \leq \frac{\E \exp(t p D_t)}{ \exp(p c t)}=\frac{\E X_t^p}{ \exp(pct)}.
\]
Corollary \ref{cor:LDP} below implements this idea to asymptotically contractive $u_t$. Recent results \cite{Ebe16, EGZ16} have shown that a large class of diffusion processes, for example over-damped Langevin processes with a convex-at-infinity potential, are asymptotically contractive.

\subsection{Paper arrangement and preliminaries}
The remainder of this paper is arranged as follow. In Section \ref{sec:poly}, Theorem \ref{thm:unstable} demonstrates that an unstable damping leads to polynomial tails. As an example, Section \ref{sec:bivariate} considers the affine damping in Figure \ref{fig:unstable}, where the exact polynomial order of the tail can be found. As another example, Section \ref{sec:LDP} employs our framework to setup a large deviation bound for long time average. Section \ref{sec:exp} discusses the scenario where the damping is nonnegative and can take value zero. Theorem \ref{thm:exponential} illustrates the necessary conditions that lead to exponential tails, while Proposition \ref{prop:weakinter} considers more general scenarios. Strictly positive  damping leading to Gaussian tails is not difficult to show and might has been proven before. But for self-containedness, we give a short proof in Section \ref{sec:Gaussian}. Lastly, Section \ref{sec:condGauss} discusses how to apply our framework to more general conditional Gaussian systems of type \eqref{sys:condGauss}.   

In order to focus on the delivery of the main ideas, we only provide the most important  arguments in our discussion. \textbf{Most technical verifications are allocated  in the appendix.}

In this paper, we use $\|a\|$ to denote the $l_2$ norm of a  vector $a$, $\langle a, b\rangle$  to denote the inner product of $a$ and $b$. $\langle \pi, f\rangle=\int f(u)\pi(du)$ is the average of $f$ under the equilibrium measure $\pi$. We denote the generator of process $(X_t, u_t)$ as $\mathcal{L}$, which can be written explicitly as below for any $C^2$ function $f$, 
\[
\mathcal{L} f(x,u)=-\langle bx, \nabla_x f\rangle+\langle h, \nabla_u f\rangle+\frac12 \text{tr}(\sigma_x^2\nabla_x^2  f)+\frac12 \text{tr}(\nabla_u^2 f). 
\]
In above, $\nabla_x$ and $\nabla_u$ are the gradients with respect to variables $x$ and $u$, and $\nabla^2_x$ and $\nabla^2_u$ are the corresponding Hessian matrices.  We can also define the Carre du champ operator using $\mathcal{L}$ \cite{BGL13}:
\begin{equation}
\label{eqn:carre}
\Gamma(f,g)=\frac{1}{2}(\mathcal{L}(fg)-f\mathcal{L}g-g\mathcal{L}f)=\frac12\sigma_x^2 \langle  \nabla_x g,  \nabla_x f\rangle+\frac12 \langle  \nabla_u g,  \nabla_u f\rangle. 
\end{equation}
Obviously $\Gamma$ is bilinear, symmetric and positive. We will also write $\Gamma(g):=\Gamma(g,g)$ for simplicity.  One important arithmetic property of $\Gamma$ is the following chain rule of the generator 
\begin{equation}
\label{eqn:chain}
\mathcal{L} \varphi(g)=\dot{\varphi}(g)\mathcal{L} g+\ddot{\varphi}(g)\Gamma (g).
\end{equation}
The derivation of the formula above and more properties of $\Gamma$ can be found in \cite{BGL13}. Also, it is worth noting that in our discussion below, we are often concerned with functions of only one variable, that is $f(x,u)=f(x)$ or $f(x,u)=f(u)$. Then some parts of the formulas above will vanish. 

The moment function $\|x\|^p$ will naturally be of interest in our discussion. Unfortunately it is not $\mathcal{C}^2$ at the origin when $p<2$, so $\mathcal{L}$ cannot be applied. To remedy this, we will often use 
\begin{equation}
\label{eqn:CEp}
\CE_p(x):=\frac{\|x\|^{p+2}}{1+\|x\|^2}+1
\end{equation}
as a surrogate, which is also used in \cite{BGM10, MT16RMS}. Its equivalence with $\|x\|^p$ is established below:
\begin{lem}
\label{lem:Eq}
 For any $p>0$, $\CE_p(x)$ in \eqref{eqn:CEp} is equivalent to $\|x\|^p$ in the following sense:
\[
\frac{1}{2}(\|x\|^p+1)\leq \CE_p(x)\leq \|x\|^p+1.
\]
Moreover, for any $\delta>0$, there is a $C_\delta>0$ such that 
\[
-(pb(u)+\delta|b(u)|+\delta) \CE_p(x)-C_\delta (|b(u)|+1) \leq \mathcal{L}\CE_p(x,u)\leq -(pb(u)-\delta |b(u)|-\delta) \CE_p(x)+C_\delta (|b(u)|+1).
\]
At here and below, we use $\CE_p(x,u):=\CE_p(x)$ with the generator to emphasize that $\mathcal{L}\CE_p(x,u)$ depends on both $x$ and  $u$. 
\end{lem}
Another useful result is a comparison principle for systems of form \eqref{sys:dyd}. 
\begin{prop}
\label{prop:compare}
Suppose there is another process $Y_t\in\reals^{d_X}$ driven by the same $u_t$ process:
\[
\rmd Y_t=-b'(u_t)Y_t \rmd t+\sigma_x \rmd W_t,\quad Y_0=X_0,
\]
where $b(u)\geq b'(u)$ for all $u$. Then $\E \|Y_t\|^{2p}\geq \E \|X_t\|^{2p}$ for all integer $p$.
\end{prop}

\section{Polynomial tails from unstable dampings}
\label{sec:poly}
Our first result shows that if  the damping is unstable, that is $b(u_*)<0$ for some $u_*$, then $X_t$ in \eqref{sys:dyd} will have polynomial tails. This involves two parts, showing $\E \|X_t\|^p<\infty$ when $p>0$ is small enough, and $\E \|X_t\|^p=\infty$ when $p$ is large enough.

To establish the lower bound, that is $\lim_t \E \|X_t\|^p=\infty$ for a large $p$, it suffices to assume some general regularity and growth conditions on $b$ and $h$. 
\begin{aspt}
\label{aspt:regular}
Suppose the following holds for all $y$, where  $C>0, m\geq 2$ are constants,  $M_y$ is a constant that may depend on $y$:
 \[
 \|h(x)\|\leq C\|x\|^{m-1}+C,\quad |b(x)-b(y)|\leq M_y \|x-y\|+M_y \|x-y\|^{m}.
 \]
\end{aspt}
To establish the upper bound, we need in addition that  $u_t$ is asymptotically contractive:
\begin{defn}
\label{defn:lip}
Given two distributions $\mu$ and $\nu$, we use $d(\mu,\nu)$ to denote the  Wasserstein-1 distance between $\mu$ and $\nu$, generated by the $l_2$ norm. Let $P_t^u$ denote the distribution of $u_t$ with $u_0=u$. We say $u_t$ is asymptotically contractive if there are constants $C_\gamma,\gamma>0$ such that 
\[
d(P_t^u, P_t^v)\leq C_\gamma\exp(-\gamma t )\|u-v\|
\]
holds for all $u,v$ and $t$. 
\end{defn}
Recent results \cite{Ebe16, EGZ16} have shown that a wide range of SDE are asymptotically contractive. For example, if $u_t$ follows the overdamped Langevin dynamics, that is $h(u)=-\nabla H(u)$ in \eqref{sys:dyd}, and the potential $H$ is strictly convex outside a bounded region, then $u_t$ is asymptotically contractive.
If $u_t$ is a stable OU process, this assumption holds naturally. 

The general statement of unstable damping leads to polynomial tails is given below.
\begin{thm}
\label{thm:unstable}
Under  Assumption \ref{aspt:regular}, suppose that $b(u^*)<0$ for a certain $u^*$.
\begin{enumerate}[1)]
\item If $\sigma_x>0$, then 
\[
\lim_{t\to\infty} \E \|X_t\|^p=\infty,\quad \text{for sufficiently large }p. 
\]
\item If $u_t$ is asymptotically contractive, $b$ has Lipschitz constant $\|b\|_{Lip}$, and the average damping $\langle \pi, b\rangle>0$,  then for any $p$ such that 
\[
p\langle\pi,b\rangle- \frac12 p^2 C_\gamma^2\gamma^{-2} \|b\|^2_{Lip}>0, 
\]
we have 
\[
\limsup_{t\to\infty} \E \|X_t\|^p<\infty. 
\]
\item Assuming the conditions of 2), if in addition $\sigma_x>0$, and $h$ preserves energy, that is for some constants $\lambda>0$ and $M_\lambda>0$,
\[
\langle h(u), u\rangle \leq -\lambda \|u\|^2+M_\lambda,
\] 
then $(X_t, u_t)$ is geometrically ergodic.
\end{enumerate} 
\end{thm}
The proof comes as a combination of  the results from the next three subsections. The complete proof can be found in the appendix.

Before we move on, we give a quick remark on the average damping condition $\langle \pi, b\rangle>0$.
This is a necessary condition. In the simplified case $\sigma_x=0$ and $X_0=x_0$, 
\[
\E \|X_t\|^p=\|x_0\|^p \E \exp\left(-p\int^t_0 b(u_s)\rmd s\right). 
\]
By Jensen's inequality, the long time damping effect on $\|X_t\|^p$ can be bounded by
\[
\E \exp\left(-p\int^t_0 b(u_s)\rmd s\right)\geq \exp\left(-p\int^t_0 \E b(u_s)\rmd s\right)\overset{t\to \infty}{\approx}  \exp\left(-p t\langle \pi,  b\rangle \right).
\]
So in order for $\E \|X_t\|^p$ to be stable, $\langle \pi, b\rangle$ needs to be positive.  On the other hand, Jensen's inequality provides only one side of the estimate. In fact, when $b$ is not strictly positive, the long time damping effect, $\E \exp\left(-p\int^t_0 b(u_s)\rmd s\right)$, does not scale as $\exp(-c p t)$ for large $p$. This is the main mechanism behind the extreme events and heavy tails.

\label{sec:lyap}
\subsection{Building  Lyapunov functions}
In order to show $\mathbb{E}\|X_t\|^p$ is bounded uniformly in time, we will try to find a  Lyapunov function $V(x,u)\approx \|x\|^p$, such that for some $\rho , k_v>0$, when applying the generator $\mathcal{L}$ of  \eqref{sys:dyd}
\begin{equation}
\label{tmp:up}
\mathcal{L} V(x,u)\leq -\rho  V(x,u) + k_v.
\end{equation}
Then applying  Gronwall's inequality and Dynkin's formula, we have 
\[
\mathbb{E}V(X_t, u_t)\leq e^{-\rho  t}\mathbb{E}V(X_0,u_0)+k_v/\rho .
\] 
Conversely, in order to show $\mathbb{E}\|X_t\|^p\to \infty$ for $t\to \infty$, it suffices to find a function  $U(x,u)\approx  \|x\|^p$,  such that for some $\rho , k_v>0$
\[
\mathcal{L} U(x,u)\geq \rho  U(x,u) - k_v. 
\]

The key to this method is finding the proper $V$ and $U$. One naive attempt is letting $U$ or $V$ to be $\|x\|^p$. However, this will not be sufficient, since for $p\geq 2$, 
\[
\mathcal{L}\|x_t\|^p=-pb(u_t) \|x_t\|^p+\frac12\sigma_x^2p(p-2+d_X)\|x_t\|^{p-2}.  
\]
An inequality like \eqref{tmp:up} does not hold because of  the appearance of $b(u_t)$.

The main idea here is to look for a function that is the  product of two parts, one part is a potential that depends on $u$, the other part is roughly the  moment of $x$:
\begin{lem}
\label{lem:lyap}
Fix  $q>0$ and $\delta>0$. Assume there are  functions  $\CE_q>0, f>0$ such that for some $C_\delta>0$ and $\rho$
\begin{equation}
\label{eqn:lyapup}
\begin{gathered}
\mathcal{L} \CE_q(x,u)\leq -(qb(u)-\delta|b(u)|-\delta) \CE_q(x,u)+C_\delta(1+|b(u)|),\\
\mathcal{L} f(u)-(qb(u)-\delta|b(u)|-\delta) f(u)\leq -\rho f(u),
\end{gathered}
\end{equation}
then $V(x,u)= f(u) \CE_q(x,u)$ satisfies: $\mathcal{L}V(x,u)\leq -\rho V(x,u)+C_\delta(1+|b(u)|) f(u).$ 

The converse is also true. If there are  functions  $\CE_q>0, g>0$ such that 
\begin{equation}
\label{eqn:lyaplow}
\begin{gathered}
\mathcal{L} \CE_q(x,u)\geq -(qb(u)+\delta|b(u)|+\delta)\CE_q(x,u)-C_\delta(1+|b(u)|),\\
 \mathcal{L} g(u)-(qb(u)+\delta|b(u)|+\delta) g(u)\geq \rho g(u),
 \end{gathered}
\end{equation}
then $U(x,u)=g(u) \CE_q(x,u)$ satisfies: $\mathcal{L}U(x,u)\geq \rho U(x,u) -C_\delta(1+|b(u)|) g(u).$ 

\end{lem}
\begin{proof}
By the product rule,  the generator of $V(x,u)$ is 
\[
\mathcal{L}V(x,u)=\CE_q(x,u)[\mathcal{L}f(u)]+[\mathcal{L}\CE_q(x,u)]f(u)\quad (\text{or the  similar version with } g).
\]
We have the claim once the conditions are plugged in. 
\end{proof}
For our purpose, we will let $\CE_q=\frac{\|x\|^{q+2}}{1+\|x\|^2}+1$ as in \eqref{eqn:CEp}. This choice satisfies the requirement of Lemma \ref{lem:lyap}, and $\CE_q$ is equivalent to $\|x\|^q$ by Lemma \ref{lem:Eq}. We don't use $\|x\|^q$ directly because it is not $C^2$ when $q<2$. 

If we can find a regular $f$ that satisfies \eqref{eqn:lyapup}, we can show $\limsup_{t\to \infty}\E \|X_t\|^p$ is finite for $p<q$. Conversely, with a regular $g$ that satisfies \eqref{eqn:lyaplow}, we can show $\lim_{t\to \infty} \E \|X_t\|^p=\infty$ for $p>q$. This is proved by the following lemma:
\begin{lem}
\label{lem:gronwall}
Suppose there is a function $f$ that satisfies \eqref{eqn:lyapup} with a $\rho>0$, and
\[
\limsup_{t\to \infty}\E (1+|b(u_t)|)f(u_t)<M_0,\quad  \limsup_{t\to \infty}[\E f(u_t)^{-\frac1\alpha}]^{\alpha}<M_\alpha,
\]
 for any $\alpha>0$ with an appropriate $M_\alpha$, then 
\[
\limsup_{t\to\infty}\E\|X_t\|^p<\frac2\rho C_\delta M_0 M^{\frac pq}_{\frac{q-p}{p}},\quad \forall p<q.
\]
Conversely, suppose $\sigma_x> 0$ and there is a function $1\geq g>0$ that satisfies \eqref{eqn:lyaplow} with a $\rho>0$,  and
\[
\limsup_{t\to \infty}\E (1+|b(u_t)|)<M_0,
\]
then $\lim_{t\to\infty}\E\|X_t\|^p= \infty $  for any $p>q$. 
\end{lem}
\subsection{Lower bound: constructive verification}
Based on Lemmas \ref{lem:lyap} and \ref{lem:gronwall}, in order to show that $\lim_{t\to\infty}\E \|X_t\|^q=\infty$  for a large $q$, it suffices to find a positive function $g\leq 1$ such that \eqref{eqn:lyaplow} holds.

Let $\eta=q^{-1}\log g$, it is well defined.  Then by the chain rule formula \eqref{eqn:chain},   \eqref{eqn:lyaplow} is equivalent to 
\begin{align*}
(qb(u)+\delta|b(u)|+\delta+\rho)\exp(q\eta)&\leq \mathcal{L}\exp(q\eta)\\
&=q \exp(q\eta)\langle h, \nabla_u \eta\rangle+\frac12 \exp(q\eta)( q \text{tr}(\nabla_u^2 \eta)+q^2\|\nabla_u \eta\|^2 ).
\end{align*}
In other words, we need to find an $\eta\leq 0$ such that 
\begin{equation}
\label{eqn:etalow}
\left(\text{tr}(\nabla_u^2 \eta)+q\|\nabla_u \eta\|^2\right)+2\langle h,\nabla_u \eta\rangle\geq 2b+2q^{-1}(\delta+\delta|b(u)|+\rho).
\end{equation}
This can be done by an explicit construction, as long as $b$ and $h$ are regular as in Assumption \ref{aspt:regular}.
\begin{lem}
\label{lem:unstablelower}
Suppose $b(u^*)<0$.  Under Assumption \ref{aspt:regular}, by choosing a sufficiently small $c>0$ and sufficiently large $q$, \eqref{eqn:etalow} holds with 
\[
\eta(u)=-c\|u-u^*\|^{m}\leq 0. 
\]
\end{lem}
\begin{proof}
By our assumption, there is an $M$ such that
\[
b(u)\leq b(u^*)+M \|u-u^*\|+M\|u-u^*\|^{m}, \quad \forall u,
\]
and $|b(u)|\leq |b(u^*)|+M \|u-u^*\|+M\|u-u^*\|^{m}$. 
Then notice that 
\[
\text{tr}(\nabla^2_u \eta)=-m(m-2+d_u)c\|u-u^*\|^{m-2}\geq 0,
\]
\[
\nabla_u \eta= -mc(u-u^*)\|u-u^*\|^{m-2},\quad \|\nabla_u \eta\|^2=m^2c^2\|u-u^*\|^{2m-2}.
\]
Under Assumption \ref{aspt:regular}, by Cauchy Schwarz inequality, we can increase $M$ such that
\[
\langle h, \nabla_u \eta\rangle \geq -cM \|u-u^*\|^{2m-2}-cM\|u-u^*\|^{m-1}.
\]
So in combine, to show \eqref{eqn:etalow} it suffices to show, 
\begin{align*}
&(q m^2c^2-2cM)\|u-u^*\|^{2m-2}+(-2b(u^*)-2q^{-1}(\delta+\delta |b(u^*)|+\rho))\\
&\geq m(m-2+d_u)c\|u-u^*\|^{m-2}+2cM\|u-u^*\|^{m-1}\\
&\qquad +M(1+2q^{-1}\delta)\|u-u^*\|^m+(1+2q^{-1}\delta)M\|u-u^*\|. 
\end{align*}
By Young's inequality and $m\geq 2$,  this can be achieved by a sufficiently large $q$ and small $c$.
\end{proof}

\subsection{Upper bound: solution from the Feynman Kac formula}
To find a $f$ that satisfies \eqref{eqn:lyapup},  we let $\theta=q^{-1}\log f$. Then similar to the derivation of  \eqref{eqn:etalow}, we find that \eqref{eqn:lyapup} is equivalent to 
\begin{equation}
\label{eqn:theta}
\mathcal{L}\theta(u)\leq \btilde(u)-q^{-1}(\rho+\delta) -\frac{1}{2}q\|\nabla_u \theta(u)\|^2,\quad \btilde(u):=b(u)-q^{-1}\delta|b(u)|.
\end{equation}
Directly solving \eqref{eqn:theta} is challenging, since it involves a nonlinear term $\|\nabla_u \theta(u)\|^2$. 
Here the idea is that we look for $\theta$ that is Lipschitz, so with a certain constant $M$,  $\frac{1}{2}\| \nabla_u \theta(u)\|^2\leq M$. Then for \eqref{eqn:theta} to hold, it suffices to solve a linear problem:
\[
\mathcal{L}\theta(u)\leq \btilde(u)-q^{-1}(\rho+\delta) -qM. 
\]
To solve this, we recall the formula for  Cauchy problems. Given a specific $\btilde(u)$, the solution of 
\[
\mathcal{L}\theta(u)=\btilde(u)
\]
exists if  and only if $\langle \pi, \btilde\rangle=0$,  and $\theta$ is  given by the following Feynman Kac's formula
\[
\theta(u)=-\int^\infty_0 \E^u \btilde(u_t) \rmd t. 
\]
$\E^u$ here denotes the conditional expectation with $u_0=u$. For self-completeness, we verify this fact in Lemma \ref{lem:cauchy}. For our purpose, it is natural to try
\begin{equation}
\label{eqn:buildtheta}
\theta(u)=-\int^\infty_0 \E^u (\btilde(u_t)-\langle \pi, \btilde\rangle) \rmd t.
\end{equation}
Then to verify \eqref{eqn:theta}, we simply need 
\[
q^{-1}(\rho+\delta)+\frac{1}{2}q\|\nabla_u \theta(u)\|^2  \leq \langle \pi, \btilde\rangle. 
\]
Note that $\rho+\delta$ can be an arbitrarily small positive number, and 
\[
\langle \pi, \btilde\rangle=\langle \pi, b\rangle-q^{-1}\delta\langle \pi, |b|\rangle>0
\] 
by our assumption, so it suffices to verify that $\|\nabla_u \theta(u)\|$ is bounded globally. We would assume the following assumption for $u_t$. 

\begin{lem}
\label{lem:linear}
Assume that $u_t$ satisfies the asymptotic Lipschitz contraction, and define $\theta$ as in \eqref{eqn:buildtheta}, then 
\[
\|\nabla_u \theta(u)\|\leq C_\gamma \gamma^{-1}\|\btilde\|_{Lip}.
\]
And if $\btilde=b-q^{-1}\delta |b|$, then $\|\btilde\|_{Lip}\leq (1+q^{-1}\delta)\|b\|_{Lip}$. 
\end{lem}
\begin{proof}
Note that 
\begin{align*}
\|\nabla_u\mathbb{E}^u \btilde(u_t)\|&=\sup_{\|v\|=1}\lim_{\epsilon\to 0}\epsilon^{-1}[\mathbb{E}^{u+\epsilon v} \btilde(u_t)-\mathbb{E}^u \btilde(u_t)]\\
&\leq \sup_{\|v\|=1}\lim_{\epsilon\to 0}\epsilon^{-1}\|\btilde\|_{Lip} d(P^{u+\epsilon v}_t, P_t^u)\\
&\leq \sup_{\|v\|=1}\lim_{\epsilon\to 0}\|\btilde\|_{Lip} C_\gamma \|v\|\exp(-\gamma t)=C_\gamma\|\btilde\|_{Lip}\exp(-\gamma t). 
\end{align*}
Therefore
\begin{equation}
\label{eqn:upper}
\left\|\nabla_u \theta (u)\right\|=\bigg\|\int^\infty_0 \nabla_u \mathbb{E}^u \btilde(u_t)\rmd t\bigg\|\leq C_\gamma \|\btilde\|_{Lip}\int^\infty_0 \exp(-\gamma t)\rmd t=C_\gamma\gamma^{-1}\|\btilde\|_{Lip}. 
\end{equation}
Finally note that $\btilde$ is Lipschtiz as long as $b$ is:
\[
|\btilde(x)-\btilde(y)|\leq |b(x)-b(y)|+q^{-1}\delta||b(x)|-|b(y)||\leq (1+q^{-1}\delta)|b(x)-b(y)|. 
\]
\end{proof}

\subsection{Example: unforced SPEKF}
\label{sec:bivariate}
The general discussion above may look technical in the first reading. Here we explain the intuition using a simple example. Consider the unforced SPEKF model \eqref{eqn:2OU} with an affine damping function 
\begin{equation}
\label{sys:1dim}
\begin{gathered}
\rmd X_t=-b(u_t+m_u) X_t \rmd t+ \rmd W_t,\\
\rmd u_t=-\gamma u_t \rmd t+\rmd B_t.
\end{gathered}
\end{equation}
For simplicity, we also only consider moments of order $q\geq 2$, so the generator can be directly apply to $|x|^q$. This eliminates the perturbation terms of order $\delta$ in Lemma \ref{lem:Eq} and the followup discussion.  The $\theta(u)$ in Lemma \ref{lem:cauchy}, by letting $\delta=0$, is given by
\[
\theta(u)=-b \int^\infty_0 \rmd t \E^u  u_t  =  -bu\int^\infty_0\exp(-\gamma t)\rmd t=-\frac{b u}{\gamma}. 
\]
From this, we see that Lemma \ref{lem:linear} is sharp, since $u_t$ is asymptotically contractive with $C_\gamma=1$, and $\gamma=\gamma$. This suggests us to use $f(u):=\exp( -qb u/\gamma)$ in Lemma \ref{lem:lyap}. If fact $g$ can be chosen as the same. Simply note that
\[
\mathcal{L} f(u)=(qb u  +\tfrac12 q^2 b^2/\gamma^2) f(u).
\]
If we let $V(x,u)=|x|^q f(u)$ with $q\geq 2$, note that 
\[
\mathcal{L} V(x,u)=|x|^q\mathcal{L} f(u)+\mathcal{L} |x|^q f(u),\quad \mathcal{L}|x|^q=-qb(u+m_u)|x|^q+\frac12q(q-1)|x|^{q-2},
\]
so
\begin{align*}
 \mathcal{L} V(x,u)= (\tfrac12 q^2 b^2/\gamma^2-qbm_u ) V(x,u)+\frac12q(q-1)|x|^{q-2} f(u). 
\end{align*}
By Young's inequality, for any $\delta>0$, there is a $C_\delta>0$ so that
\[
(\tfrac12 q^2 b^2/\gamma^2-qbm_u -\delta) V(x,u)-C_\delta f(u)\leq  \mathcal{L} V(x,u)\leq (\tfrac12 q^2 b^2/\gamma^2-qbm_u +\delta) V(x,u)+C_\delta f(u). 
\]
By Lemma \ref{lem:gronwall}, this means that $\limsup_{t\to \infty}\|X_t\|^q$ is infinite if $q>q_0$, and is finite if $q<q_0$. The threshold here is given by
\[
q_0=\frac{2m_u\gamma^2}{b}.
\]

\subsection{Example: large deviation bound}
\label{sec:LDP}
As another example, we demonstrate how to apply our framework to show the deviation of long time average of $u_t$ is sub-Gaussian.
\begin{cor}
\label{cor:LDP}
Assume that $u_t$ is asymptotically contractive, and it is exponentially integrable, that is
\[
\limsup_{t\to \infty}\E \exp(\alpha \|u_t\|)<\infty,\quad \forall \alpha. 
\]
Then  for any $\delta>0$, the following large deviation bound holds for certain $M_\delta$
\[
\mathbb{P}\left(\frac{1}{t}\int^t_0 b(u_s)\rmd s-\langle \pi, b\rangle>D_M c\right)\leq M_\delta \exp\left(-(\tfrac12 c^2-\delta) D_M t\right),\quad \forall t>0,
\] 
where $D_M=C_\gamma^2\gamma^{-2}\|b\|^2_{Lip}$. 
\end{cor}
\begin{proof}
Consider $X_t=\exp\left(\int^t_0 (b(u_s)-\langle \pi, b\rangle) ds \right)$ as defined in \eqref{eqn:XtLDP}, it follows the ordinary differential equation:
\[
dX_t=(b(u_t)-\langle \pi, b\rangle)X_tdt.
\]
Since there is no diffusion term in $X_t$, we consider the function
\[
V(x,u):=x^q\exp(q \theta (u)),
\]
where $\theta(u)=\int^\infty(\E^ub(u_t)-\langle \pi, b\rangle)dt$  satisfies the following by Lemmas \ref{lem:cauchy} and \ref{lem:linear}
\[
\mathcal{L}\theta=\langle \pi, b\rangle-b,\quad  \|\nabla_u \theta\|\leq C_\gamma \gamma^{-1}\|b\|_{Lip}.
\]
Applying the generator to $V$, we find
\begin{align*}
\mathcal{L} V(x,u)&=(\mathcal{L} x^q) \exp(q\theta(u))+x^q \mathcal{L} \exp(q\theta(u))\\
&= \left(q (b-\langle \pi, b\rangle) +q \mathcal{L}\theta+\frac12 q^2\|\nabla_u \theta\|^2\right)V(x,u)\leq \frac12 q^2 D_M V(x,u).
\end{align*}
So by Dynkin's formula,
\[
\E X^q_t\exp(q \theta (u_t))\leq \exp\left(\frac 12 q^2 t D_M\right)\E \exp(q \theta (u_0)). 
\]
For any $p<q$, by H\"{o}lder's inequality, 
\begin{align*}
\E X_t^p&\leq \bigg(\E X^q_t\exp(q \theta (u_t))\bigg)^{\frac{p}{q}}\bigg(\mathbb{E}\exp(p \theta (u_t))^{\frac{-q}{q-p}}\bigg)^{\frac{q-p}{q}}\\
&\leq \exp( \tfrac12qp t D_M) \left(\E \exp(q \theta (u_0)\right)^\frac{p}{q}\bigg(\mathbb{E}\exp(p \theta (u_t))^{\frac{-q}{q-p}}\bigg)^{\frac{q-p}{q}}\leq  \exp( \tfrac12qp t D_M)M_{p,q}.
\end{align*}
The constant $M_{p,q}$ exists because we assume $\limsup_{t\to \infty}\E \exp(\alpha \|u_t\|)<\infty$, and $\theta$ is Lipschitz. Therefore if we let
\[
D_t:=\frac{1}{t}\int^t_0 b(u_s)\rmd s-\langle \pi, b\rangle,
\]
then
\[
\Prob(D_t \geq D_Mc) \leq \frac{\E \exp(t p D_t)}{ \exp(p t D_Mc)}=\frac{\E X_t^p}{ \exp(p D_Mt c)}\leq \exp( (\tfrac12q-c) p t D_M)M_{p,q}.
\]
We pick $p=c,q=c+\frac{2\delta}{c}$ and find our claim.
\end{proof}

\section{Exponential tails from  nonnegative dampings}
\label{sec:exp}
This section shows that nonnegative dampings  lead to exponential or weaker tails. 
\begin{thm}
\label{thm:exponential}
Suppose the following hold:
\begin{itemize}
\item $u_t$ is asymptotically contractive.
\item  The damping function satisfies $b(u)\geq 0$ and $b(u)=0$ when $\|u-u_*\|\leq \epsilon$ for some $u_*$.  Also there is a Lipschitz function $b'$ such that $b\geq b'\geq 0$. 
\item The dynamics of $u_t$ dissipates the energy centered at $u_*$, that is there are $\lambda, M_\lambda>0$ so that 
\end{itemize}
\[
\langle u-u_*, h(u)\rangle \leq -\lambda \|u-u_*\|^2+M_\lambda.
\]
Then $X_t$ has exponential-like tails. In particular
\[
\lim_{p\to \infty}\lim_{t\to \infty}\frac{\log \mathbb{E}\|X_t\|^{2p}}{p\log p}= 2.
\]
\end{thm}
Again, this result consists of two parts. The upper bound comes from Proposition \ref{prop:upperbound}. The lower bound comes from  the combination of Proposition \ref{prop:weakinter} and Corollary \ref{cor:dissi}. One can find the detailed verification at the end of \ref{sec:appendexp}.

Theorem \ref{thm:exponential} doesn't consider the delicate case where $b(u)=0$ only at a single point. This was mentioned in Theorem \ref{thm:simple} as case (iii).  We only provide a lower bound in Proposition \ref{prop:weakinter}, indicating the tail is strictly heavier than Gaussian. This is already useful in practice. Also note that the statement of Theorem \ref{thm:simple} is rigorous, as we only claim that the distribution is between exponential and Gaussian.

\subsection{Upper bound}
 As a matter of fact, it is relatively easy to see that a process with nonnegative damping has sub-exponential tails. By the comparison principle Proposition \ref{prop:compare}, we only need to consider $b$ that is Lipschitz. 
\begin{prop}
\label{prop:upperbound}
Suppose $b\geq 0$, and the $\theta$ in \eqref{eqn:buildtheta} is well defined, with the Carre du champ $\Gamma(\theta)$ defined in \eqref{eqn:carre} bounded, and the following integrability condition holds for any $\alpha\in\reals$ 
\[
\limsup_{t\to \infty} \E \exp( \alpha\theta(u_t))<\infty,\quad \limsup_{t\to \infty} \E b(u_t)\exp( \alpha\theta(u_t))<\infty. 
\]
Then $X_t$ has sub-exponential tails. In particular, for any $\beta\in \reals^{d_X}$ such that
\[
\tfrac12\|\beta\|^2 \|\nabla_u \theta\|^2+\tfrac12 \sigma_x^2\| \beta\|^2-\|\beta\| \langle \pi, b\rangle<0,
\]
then
\[
\limsup_{t\to\infty}\mathbb{E}\exp \langle\beta, X_t\rangle <\infty. 
\]
In particular, this indicates that  
\[
\limsup_{p\to \infty}\limsup_{ t\to \infty}\frac{\log \E\|X_t\|^{2p}}{p\log p}\leq  2.
\]
\end{prop}

\subsection{Lower bound}
To show the lower bound requires additional work. First let us define the set of damping functions that can yield approximately exponential tails.
\begin{defn}
\label{defn:Am}
We say  a pair of functions $(b,h)\in \mathcal{A}_m$  if there are $g_1,\ldots, g_m$  on $\mathbb{R}^{d_u}$ such that
\begin{enumerate}[(1)]
\item $b\geq 0$, and $b(u_*)=0$ for certain $u_*$. 
\item $g_1$ satisfies the level-1 constraint $\Gamma(g_1)\geq b$.
\item $g_k$ satisfies the level-k constraint $\Gamma(g_k)+\mathcal{L} g_{k-1}\geq 0$, for $k=2,\ldots m$.
\item $\mathcal{L}g_m\geq -M$ for a constant $M$.
\item There  are constants $M_0$ and $M_1$
\[
 G_p(u)=\sum_{k=1}^m p^{\frac{1}{2^k}} g_k(u)\leq  \sqrt{p} M_0, \quad \E G_p(u_0)\geq -\sqrt{p}M_1. 
\]
\item Alignment condition: for all $j,k\leq m$, $\Gamma(g_j,g_k)\geq 0$.
\end{enumerate}
\end{defn}

The long time damping effect from $b(u_t)\in \mathcal{A}_m$ is  revealed by the following lemma. The main message is that $b(u_t)$ creates a weaker long time damping when applied to higher moments of $X_t$. 
\begin{lem}
\label{lem:weakdampgen}
Suppose $(b,h)\in \mathcal{A}_m$, then the following holds under the invariant measure for $p\geq 1$ :
\[
 \mathbb{E} \exp\left(-2p\int^t_0 b(u_s)\rmd s\right)\geq \exp(-2 p^{\frac{1}{2^m}}M t-2\sqrt{p}M_0-2\sqrt{p}M_1) . 
\]
\end{lem}
Sub-exponential tails come as a result of this weak long time damping. 
\begin{prop}
\label{prop:weakinter}
Suppose the damping and the drift of $u_t$, $(b,h)$, belongs to  $\mathcal{A}_m$ as in Definition \ref{defn:Am}, then under the equilibrium measure,
\[
\liminf_{p\to\infty}\liminf_{t\to \infty}\frac{\log \mathbb{E}\|X_t\|^{2p}}{p\log p}\geq 2-\frac{1}{2^m}.
\]
In other words, $X_t$ has a tail between exponential and Gaussian. 
\end{prop}

\subsection{Energy dissipation}
If $(b,h)$ is in $\mathcal{A}_m$ for all $m$, then Proposition \ref{prop:weakinter} indicates the higher moments of $\|X_t\|$ behaves very much the same as the exponential distribution. In this section we show that this will be the case under the conditions of Proposition \ref{prop:weakinter}.

\begin{lem}
\label{lem:dissi}
Suppose $b(u_*)=0$, and for some $m\in \mathbb{Z}^+, C>0$,
\begin{equation}
\label{eqn:bbounded}
b(u)\leq C\|u-u_*\|^{2^{m+1}-2}.
\end{equation}
Suppose also the energy centered at $u_*$ is dissipative under the drift $h$,  so for some $\lambda,M_\lambda>0$, 
\begin{equation}
\label{eqn:dissipative}
\langle u-u_*, h(u)\rangle\leq -\lambda\|u-u_*\|^2 +M_\lambda.
\end{equation}
Then $(b,h)\in \mathcal{A}_m$.
\end{lem}

\begin{cor}
\label{cor:dissi}
Suppose  $b(u)=0$ when $\|u-u^*\|\leq \delta$,  and $b$ has polynomial growth for some $n$
\[
b(u)\leq D\|u-u_*\|^{n}.
\]
Suppose  the energy centered at $u_*$ is dissipative under the drift $h$, so \eqref{eqn:dissipative} holds. Then $(b,h)\in \mathcal{A}_m$ for all $m$ such that $2^{m+1}\geq n+2$. In other words, $X_t$ will have an exponential like tail. 
\end{cor}
\begin{proof}
We just need to verify \eqref{eqn:bbounded} for some $C$. Simply note that when $\|u-u_*\|\geq \delta$
\begin{align*}
b(u)\leq D\|u-u_*\|^{n}&=D\delta^n \|(u-u_*)/\delta\|^{n}\\
&\leq D\delta^n \|(u-u_*)/\delta\|^{2^{m+1}-2}=D\delta^{n+2-2^{m+1}} \|u-u_*\|^{2^{m+1}-2}.
\end{align*}
And when $\|u-u_*\|<\delta$, $b(u)=0$, so \eqref{eqn:bbounded} holds automatically. 
\end{proof}

\section{Sub-Gaussian tails from strictly positive dampings}
The following analysis is rather standard. But since it is short and we want to be self-contained, we provide the details rather than finding a reference. 
\label{sec:Gaussian}
\begin{thm}
\label{thm:Gaussiantail}
Suppose $b(u)\geq b_0$ for a $b_0>0$, then if $\alpha<b_0$,
\[
\limsup_{t\to \infty}\E \exp(\alpha\|X_t\|^2)<\infty. 
\]
This leads to 
\[
\limsup_{p\to \infty}\limsup_{t\to \infty}\frac{\E \|X_t\|^{2p}}{p\log p}\leq 1.
\]
If $u_t$ dissipates the energy, that is 
\[
\langle h(u), u\rangle\leq -\lambda \|u\|^2+M_\lambda,
\] 
then if $\lambda>\alpha$, 
\[
\limsup_{t\to \infty}\E \exp(\alpha\|u_t\|^2)<\infty. 
\]
\end{thm}
\begin{proof}
Let $\CE(x)=\exp(\alpha \|x\|^2)$, with $\alpha<b_0/\sigma_x^2$. Apply the generator to it, 
\begin{align}
\notag
\mathcal{L}\CE&=-2\alpha \|x\|^2 b(u)  \CE + \sigma_x^2(\alpha d_X+2\alpha^2\|x\|^2)\CE\\
\label{tmp:gauss}
&\leq (-\delta \|x\|^2   + \sigma_x^2 d_X)\alpha\CE,
\end{align}
where $\delta=2b_0-2\alpha \sigma_x^2$.  When $\delta\|x\|^2 \leq (1+d_X)\sigma_x^2$, $\CE(x)\leq \exp((1+d_X)\alpha \sigma_x^2/\delta)$, otherwise $\mathcal{L}\CE\leq -\alpha\sigma_x^2  \CE$. Therefore 
\[
\mathcal{L}\CE\leq -\alpha\sigma_x^2  \CE+d_X \alpha \sigma_x^2  \exp((1+d_X)\alpha \sigma_x^2/\delta).
\]
So Dynkin's formula and Gronwall's inequality immediately gives us 
\[
\E\CE(X_t)\leq \exp(-\alpha\sigma_x^2 t) \E \CE(X_0)+d_X\exp((1+d_X)\alpha \sigma_x^2/\delta)<\infty. 
\]
Finally note that by Taylor expansion of $\exp(\alpha\|x\|^2)$,
\[
\|x\|^{2p}\leq p! \alpha^{-p} \exp(\alpha\|x\|^2).
\]
Apply an estimate of $\log k$ as in \eqref{tmp:logsum}, 
\[
\log \E \|X_t\|^{2p}\leq \sum_{k=1}^{p} \log k  -p\log \alpha+\log \E \exp(\alpha \|X_t\|^2)=p\log p +O(p). 
\]

A similar analysis applies to $u_t$ as well. Let $\CE(u)=\exp(\alpha \|u\|^2)$. Apply the generator to it, 
\begin{align*}
\mathcal{L}\CE&=2\alpha  \langle h(u), u\rangle  \CE + (\alpha d_u+2\alpha^2\|u\|^2)\CE\\
&\leq (-2(\lambda-\alpha) \|u\|^2   + d_u+2M_\lambda)\alpha\CE.
\end{align*}
The follow up analysis is much the same as after \eqref{tmp:gauss}.
\end{proof}

\section{General conditional Gaussian system}
\label{sec:condGauss}
The $X_t$ part in a general multivariate conditional Gaussian system \eqref{sys:condGauss} can be written as 
\[
\rmd X_t=-B(u_t)X_t \rmd t+\Sigma_X \rmd W_t,
\]
where $u_t$ is the same as in system \eqref{sys:dyd}. This formulation is different from \eqref{sys:dyd}, since $B$ is matrix-valued. 
In this section, we show how to build moment bounds for \eqref{sys:condGauss} by building surrogate damping rates as in \eqref{sys:dyd}.

To begin, we decompose $B(u_t)$ as
\begin{equation}
\label{eqn:decompose}
B(u_t)=\sum_{i=1}^n b_i(u_t)B_i,
\end{equation}
for some matrices $B_i$ and functions $b_i$. This decomposition always exists, since
\[
B(u_t)=\sum_{j,k=1}^{d_X} [B(u_t)]_{j,k}E_{j,k},
\]
where $E_{j,k}$ is the matrix with all components being zero, except the $(j,k)$-th component being $1$.  Yet this choice of decomposition can sometimes be  sub-optimal. 

With decomposition  \eqref{eqn:decompose}, let $N_i\subseteq \{1,\ldots, d_X\}$ include all indices that $B_i$ involves, so if $j\notin N_i$ then $[B_i]_{j,k}=[B_i]_{k,j}=0$ for any $k$. There are constants $m_i$ and $M_i$ such that 
\[
m_i I_{N_i}\preceq \frac12(B_i+B_i^T) \preceq M_i I_{N_i}. 
\]
Here $I_{N_i}$ is the diagonal matrix with $(j,j)$-th term being one if and only if $j\in N_i$, and being zero otherwise. With two symmetric matrices $A$ and $B$, $A\preceq B$ indicates that $B-A$ is positive semidefinite. One easy choice of $N_i$ can be $N_i=\{1,\ldots, d_X\}$ for all $i$, then $m_i$ and $M_i$ are simply the minimum and maximum eigenvalues of $\frac12(B_i+B_i^T)$. 

Consider the following two scalar value damping functions where $a\vee b=\max\{a,b\}$ and $a\wedge b=\min\{a,b\}$. 
\begin{equation}
\label{eqn:scalarb}
\bar{b}(u) =\bigwedge^{d_X}_{j=1}\sum_{i: j\in  N_i}M_i b_i\wedge m_ib_i, \quad \underline{b}(u) =\bigvee^{d_X}_{j=1}\sum_{i: j\in  N_i}M_i b_i\vee m_ib_i. 
\end{equation}
In the special case when $N_i=\{1,\cdots, d_X\}$ for all $i$, the formulation can be simplified
\[
\bar{b}(u) =\sum_{i=1}^nM_i b_i\wedge m_ib_i, \quad \underline{b}(u) =\sum_{i=1}^nM_i b_i\vee m_ib_i. 
\]

Then we have the following lemma.
\begin{lem}
\label{lem:EqGauss}
\label{lem:matrix} When applying the generator to the approximated moment $\CE_q(x)$ in Lemma \ref{lem:Eq}, for any fixed $\delta>0$, there is a constant $C_\delta$, 
\begin{align*}
-(q\underline{b}+\delta|\underline{b}|+\delta)\CE_q(x)-C_\delta (1+|\underline{b}|)\leq \mathcal{L}\CE_q(x,u)
\leq -(q\bar{b}-\delta|\bar{b}|-\delta)\CE_q(x)+C_\delta (1+|\bar{b}|).
\end{align*}
\end{lem}

Consequentially, the role of $b(u)$ in the dyadic model \eqref{sys:dyd} can be replaced by $\bar{b}$ and $\underline{b}$, when finding upper and lower bounds.  So following the proof of Theorem \ref{thm:unstable} and \ref{thm:exponential}, and \ref{thm:Gaussiantail}, we have the following corollaries:

\begin{cor}
\label{cor:upp}
Assume $u_t$ is asymptotically contractive,  $\bar{b}$  is Lipschitz and $\langle \pi, \bar{b}\rangle>0$, then 
\[
\limsup_{t\to \infty} \E \|X_t\|^p<\infty \quad \text{for some }p>0. 
\]
If furthermore $\bar{b}\geq 0$, then the tail of $\|X_t\|$ is sub-exponential. If $\bar{b}\geq b_0>0$, then the tail of $\|X_t\|$ is sub-Gaussian. 
\end{cor}

\begin{cor}
\label{cor:low}
Assume $u_t$ is asymptotically contractive, $\underline{b}$ and $h$ follow Assumption \ref{aspt:regular}.
\begin{enumerate}
\item If $\underline{b}(u_*)<0$ for some $u_*$, then for some $p>0$, $\limsup_{t\to \infty} \E \|X_t\|^p=\infty$.
\item If  $\underline{b}(u)=0$ for $u$ close to $u_*$,  the energy centered at $u_*$ is dissipative as in \eqref{eqn:dissipative}, $\underline{b}\leq D\|u-u_*\|^n$ for some $D$ and $n$,  then $\|X_t\|$ has a tail heavier than exponential. 
\end{enumerate}
\end{cor}

\section{Conclusion}
Extreme events are happening more often due to the global climate change, and they can induce heavy economic losses.
The capability to analyze and predict them is crucial for our society.  
Mathematically, they appear as strong anomalies in time series and form heavy tails in the histograms.
They are typically associated with stochastic instability caused by hidden unresolved processes.
Such instability can be modeled by stochastic dampings in conditional Gaussian models. 
This has been justified by extensive numerical experiments, while there is little theoretical understanding.
This can be problematic, since extreme events can be difficult to simulate. 

This paper closes this gap by creating a theoretical framework, 
in which  the tail density of conditional Gaussian models can be rigorously determined.
Theorem \ref{thm:unstable} shows that if the stochastic damping takes negative values, the tail is polynomial.
Theorem \ref{thm:exponential} shows that if the stochastic damping is nonnegative but takes value zero at certain points,  the tail is between exponential and Gaussian. 
These results can be generalized  to multivariate conditional Gaussian systems \eqref{sys:condGauss}, as long as certain surrogate damping rates follow the conditions in Theorems \ref{thm:unstable} and \ref{thm:exponential}. 
Moreover, we can apply the same framework  to obtain a large deviation bound for long time averaging processes. This is shown in Corollary \ref{cor:LDP}.

\section*{Acknowledgement}
The authors thank Ramon van Handel and Nan Chen for the discussion of certain parts of this paper. 
The authors also thank the anonymous referees for their extensive and detailed suggestions.
This research of A. J. M. is partially supported by the Office of Naval Research through MURI N00014-16-1-2161 and DARPA through W911NF-15-1-0636. This research of X. T. T. is supported by the National University of Singapore grant R-146-000-226-133.

\appendix

\section*{Appendix}
\setcounter{section}{1}
We allocate most of the technical verifications in this appendix.
\subsection{Some useful tools}
\begin{proof}[Proof of Lemma \ref{lem:Eq}]
The upper bound for $\CE_p$ is trivial. For the lower bound, note that if $\|x\|\leq 1$, then $\|x\|^p\leq 1$; if $\|x\|\geq 1$, then $\|x\|^{p+2}\geq \|x\|^p$. So $\|x\|^{p+2}+1\geq \|x\|^{p}$ always holds. Therefore
\[
2\CE_p-1=\frac{2\|x\|^{p+2}+\|x\|^2+1}{\|x\|^2+1}\geq \frac{\|x\|^{p+2}+(\|x\|^{p+2}+1)}{\|x\|^2+1}\geq \frac{\|x\|^{p+2}+\|x\|^p}{\|x\|^2+1}= \|x\|^p. 
\] 
The gradient and Hessian of $\CE_p$ can be computed directly:
\[
\nabla_x \CE_p= \frac{p\|x\|^{p}x}{1+\|x\|^2}+\frac{2\|x\|^{p}x}{(1+\|x\|^2)^2}.
\]
\[
\nabla^2_x\CE_p=\frac{p^2\|x\|^{p-2}xx^t}{1+\|x\|^2}+\frac{p\|x\|^{p}I}{1+\|x\|^2}-\frac{2p\|x\|^{p} xx^t}{(1+\|x\|^2)^2}+\frac{2p\|x\|^{p-2}xx^t}{(1+\|x\|^2)^2}+
\frac{2\|x\|^{p}I}{(1+\|x\|^2)^2}
-\frac{8\|x\|^{p} xx^t}{(1+\|x\|^2)^3}.
\]
Apply the generator to $\CE_p(x,u)=\CE_p(x)$, 
\[
\mathcal{L}\CE_p(x,u)=-\langle \nabla_x \CE_p(x), b(u) x\rangle+\frac{1}{2}\sigma_x^2\text{tr}(\nabla^2_x \CE_p(x)).
\]
Note that 
\[
\langle \nabla_x \CE_p(x), b(u) x\rangle=
b(u)\left(\frac{p\|x\|^{p+2} }{1+\|x\|^2}+ \frac{2\|x\|^{p+2} }{(1+\|x\|^2)^2}\right),
\]
by Young's inequality, there is a constant $C_\delta$
\begin{align*}
-(pb(u)+\delta |b(u)|+&\tfrac12\delta) \CE_p(x)-\tfrac12C_\delta |b(u)|\\
&\leq -\langle \nabla_x \CE_p(x), b(u) x\rangle\leq -(pb(u)-\delta |b(u)|-\tfrac12\delta) \CE_p(x)+\tfrac12C_\delta |b(u)|.
\end{align*}
Lastly, by Young's inequality, we can further increase $C_\delta$ so that
\[
|\text{tr}(\nabla_x^2 \CE_p(x))| =p^2O((1+\|x\|)^{p-1})\leq  \delta \CE_p(x)+ C_\delta. 
\]
In combination, we have reached our claim. 
 \end{proof}

 \begin{proof}[Proof of Proposition \ref{prop:compare}]
By the Duhamel's formula, we can write
\[
X_t=A_{0,t}X_0+ \sigma_x \int^t_{0}A_{s,t}\rmd W_s, \quad A_{s,t}:=\exp\left(-\int^t_s b(u_r)dr\right). 
\]
\[
Y_t=B_{0,t}X_0+ \sigma_x \int^t_{0}B_{s,t}\rmd W_s, \quad B_{s,t}:=\exp\left(-\int^t_s b'(u_r)dr\right). 
\]
Under our condition, $A_{s,t}\leq B_{s,t}$ a.s.. 

Conditioned on the realization of the $u_s$ process and $X_0$,  $X_t$ has a Gaussian distribution with mean being $\mu_X$ and the covariance matrix being
$\Sigma_X I$, where $\mu_X=A_{0,t} X_0$ and $\Sigma_X=\sigma_x^2\int^t_{0} A^2_{s,t}\rmd s$. Then with $Z$ being an independent $\mathcal{N}(\mathbf{0},I)$, the following holds
\begin{align*}
\E_u \|X_t\|^{2p}&=\E_u \|\mu_X+\sqrt{\Sigma_X}Z\|^{2p}\\
&=\E_u \left(\|\mu_X\|^2+2\sqrt{\Sigma_X}\langle \mu_X,Z\rangle  +\Sigma_X\|Z\|^2\right)^{p}\\
&=\sum_{k=0}^p 2^kC^k_p\sum_{j=0}^{p-k} C_{p-k}^j\|\mu_X\|^{2p-2j-2k}\Sigma_X^{j+\frac12k}\E_u\langle \mu_X,Z\rangle^k\|Z\|^{2j}.
\end{align*}
Here $C_p^k$ denotes the combinatoric number of choosing $k$ out of $p$, and the  expectation  conditioned on the realization of the $u_s$ process and $X_0$ is written as $\E_u$. 

The same formula  applies to $\E_u \|Y_t\|^{2p}$ as well, except that $\mu_X$ and $\Sigma_X$ are replaced by
$\mu_Y=B_{0,t} X_0$ and $\Sigma_Y=\sigma_x^2\int^t_{0} B^2_{s,t}\rmd s$. Note that $\|\mu_X\|\leq \|\mu_Y\|$ and $\Sigma_X\leq \Sigma_Y$ for a.s. realization of $u_s$ and $X_0$. Also note that $\E_u\langle \mu_X,Z\rangle^k\|Z\|^{2j}$ is nonzero only when $k$ is even, and it depends on $\mu_X$ only through $\|\mu_X\|$ since the distribution of $Z$ is rotation free. Therefore $\E_u \|X_t\|^{2p}\leq \E_u \|Y_t\|^{2p}$ a.s., and our claim follows by taking total expectation. 
\end{proof}

 \begin{lem}
 \label{lem:OUiscontractive}
Consider a multivariate OU process $u_t=-\Gamma u_t \rmd t+\rmd B_t$, $u_0=u$, where $B_t$ is a Wiener process of the same dimension as $u_t$.
If the real parts of the eigenvalues of the constant matrix $\Gamma$  are all strictly positive, then $u_t$ is asymptotically contractive. 
 \end{lem}
 \begin{proof}
 Consider $v_t=-\Gamma v_t\rmd t+\rmd B_t$, and $v_0=v$, where $B_t$ is the same as the one in the SDE of $u_t$. Then the distribution of $v_t$ is $P^v_t$. Yet
 \[
 d(u_t-v_t)=-\Gamma(u_t-v_t) \rmd t\quad \Rightarrow\quad  \|u_t-v_t\|\leq \|\exp(-\Gamma t) \| \|u-v\|.
 \]
 In other words $d(P^u_t, P^v_t)\leq \|\exp(-\Gamma t) \| \|u-v\|$. Because all eigenvalues of $\Gamma$ have positive real parts, so $u_t$ is asymptotically contractive. 
 \end{proof}
 
 \begin{lem}
\label{lem:cauchy}
Suppose $u_t$ is asymptotically contractive. For any Lipschitz $\psi$ 
\[
\theta(u)=-\int^\infty_0 (\mathbb{E}^u \psi(u_t)-\langle \pi, \psi\rangle) \rmd t,
\]
is well defined, and $\mathcal{L} \theta=\psi-\langle \pi, \psi\rangle$.
\end{lem}
\begin{proof}
Let $u'_t$ be an independent copy of the SDE $\rmd u'_t=h(u'_t)\rmd t+\rmd B'_t$, where $u'_0$ follows the distribution $\pi$. Then the distribution of $u_t'$ is $\pi$ by invariance. 
We consider the joint distribution of $u'_t$ and $u_t$, where $u_0=u$. We write the expectation with respect to this joint distribution as $\E$. Then
\[
\mathbb{E}^u \psi(u_t)-\langle \pi, \psi\rangle=\E \psi(u_t)-\psi(u'_t). 
\]
By the asymptotic contractiveness, we have 
\[
|\E \psi(u_t)-\psi(u'_t)|\leq \|\psi\|_{Lip}d(P_t^u, P_t^{\pi})\leq C_\gamma e^{-\gamma t} \|\psi\|_{Lip} d(\delta_u, \pi). 
\]
Therefore $\theta$ is well defined.  Next, note that
\[
\theta(u)=-\int^\infty_0 \rmd t \int \rmd z p^u_t (z) (\psi(z)-\langle \pi, \psi\rangle),
\]
where $p_t^u(z)$ is the density of $P^u_t$. Apply  Fubini's theorem, we have:
\[
\mathcal{L}\theta(u)= - \int  dz (\psi(z)-\langle \pi,\psi\rangle) \int^\infty_0 \rmd t\mathcal{L}p_t^u(z).
\]
Moreover by the Kolmogorov backward equation,
\[
\frac{\partial}{\partial t}p_t^u(z)=\mathcal{L}p^u_t(z). 
\]
Thus by $P^u_\infty=\pi$, 
\[
\mathcal{L}\theta(u)=\int dz(\psi(z)-\langle \pi,\psi\rangle) (p_0^u(z)-p_\infty^u(z))=\psi(u)-\langle \pi,\psi\rangle. 
\]
\end{proof}

 \subsection{Polynomial tails}
  \begin{proof}[Proof of Lemma \ref{lem:gronwall}]
 First we derive the upper bound. Consider the temporally inflated version of $V$, $\widetilde{V}(x,u,t):=e^{\rho t} V(x,u)$, then by Lemma \ref{lem:lyap}
\[
\mathcal{L}\widetilde{V}(x,u,t)=e^{\rho t}\mathcal{L}V(x,u)+\rho e^{\rho t} V(x,u) \leq C_\delta(1+|b(u)|)e^{\rho t} f(u). 
\]
Here we extend the definition of generator $\mathcal{L}$ so the underlying process is $(X_t, u_t, t)$. 
Thus by Dynkin's formula:
\[
\mathbb{E}\widetilde{V}(X_t,u_t, t)=\mathbb{E}\widetilde{V}(X_0,u_0, 0)+\mathbb{E}\int^t_0 \mathcal{L}\widetilde{V}(X_s,u_s, s)\rmd s
\leq \mathbb{E}V(X_0,u_0)+C_\delta \int^t_0 e^{\rho  s} \mathbb{E} (1+|b(u_s)|)f(u_s)\rmd s. 
\]
By our assumption on $f$,  
\[
\E  f(u_t) \CE_q(X_t)=\E  V(X_t,u_t)=e^{-\rho t} \mathbb{E}\widetilde{V}(X_t,u_t, t),
\]
 is bounded by $C_\delta M_0/\rho$ when $t\to \infty$.  In order to remove  the $f(u_t)$ inside the expectation,  we  apply H\"{o}lder's inequality. For any $p<q$, when $t\to \infty$, 
\begin{align*}
\E \|X_t\|^p\leq \E [2\CE_q (X_t)]^{\frac pq}&\leq 2\bigg[\E f(u_t)\CE_q(X_t)\bigg]^{\frac{p}{q}}\bigg[\mathbb{E} f(u_t)^{\frac{-p}{q-p}}\bigg]^{\frac{q-p}{q}}\leq \frac2\rho C_\delta M_0 M^{\frac pq}_{\frac{q-p}{p}}.  
\end{align*}

To prove the converse direction, let $\widetilde{U}(x,u,t):=e^{-\rho t} U(x,u)$, then
\[
\mathcal{L}\widetilde{U}(x,u,t)\geq -C_\delta e^{-\rho t} g(u)(1+|b(u)|). 
\]
Thus by Dynkin's formula:
\[
\E U(X_t,u_t)=e^{\rho t} \mathbb{E}\widetilde{U}(X_t,u_t, t)\geq e^{\rho t}\mathbb{E}U(X_0,u_0)-C_\delta \int^t_0 e^{\rho(t- s)} \mathbb{E} g(u_s)(1+|b(u_s)|)\rmd s. 
\]
Note that by ergodicity and $g\leq 1$, 
\[
\limsup_{t\to \infty} \E (1+|b(u_t)|)g(u_t)\leq\limsup_{t\to \infty} \E (1+|b(u_t)|)=M_0. 
\]
So if $\mathbb{E}U(X_0,u_0)>C_\delta M_0/\rho $, then $\E U(X_t,u_t)\to \infty$ as $t\to \infty$. 

We can generalize this using the  Markov property:
\begin{align*}
\mathbb{E}U(X_{t_0+t},u_{t_0+t})&=\mathbb{E}\mathbb{E}^{X_{t_0},u_{t_0}}U(X_{t},u_{t})\\
&\geq \mathbb{E}1_{U(X_{t_0},u_{t_0})>C_\delta M_0/\rho}\mathbb{E}^{X_{t_0},u_{t_0}}U(X_{t},u_{t})
\end{align*}
which goes to $\infty$ as $t\to \infty$ if $\Prob(U(X_{t_0},u_{t_0})>C_\delta M_0/\rho)>0$. Yet when $\sigma_x>0$, system \eqref{sys:dyd} is controllable, so  given any $\epsilon$-ball $\mathcal{B}$ centered at any point $(x',u')$, $\Prob(X_{t_0},u_{t_0}\in \mathcal{B})>0$ \cite{MSH02}. Then since $U=\CE_q g$, so $\Prob(U(X_{t_0},u_{t_0})>C_\delta M_0/\rho)>0$. 

Lastly, we note that because $g\leq 1$, for any $p>q$,
\begin{align*}
\mathbb{E}U(X_{t},u_{t})=\E g(u_t)\CE_q(X_t)\leq \E  \CE_q(X_t)\leq \E \|X_t\|^q+1\leq 
\E \|X_t\|^p+2.
\end{align*}
This gives us $\E \|X_t\|^p\to \infty$. 
\end{proof}

\begin{proof}[Proof of Theorem \ref{thm:unstable}]
By Lemmas \ref{lem:unstablelower} and \ref{lem:linear}, there are  functions  $g$ and $f$ that satisfy \eqref{eqn:lyaplow} and \eqref{eqn:lyapup}, while $\CE_q$ in Lemma \ref{lem:Eq} satisfies the conditions of Lemma \ref{lem:lyap}.  Note also that $g=\exp(q\eta)\leq 1$ and $f=\exp(q\theta)$ with $\theta$ being Lipschitz as shown by Lemma \ref{lem:linear}. In below we show that $\CE(u)=\exp(\alpha \|u\|^2)$ is a Lyapunov function, so $\limsup_{t\to\infty}\E \CE(u_t)<\infty$, and because $\theta$ has at most linear growth, so by Young's inequality for a proper constant $M$, $\CE(u)+M\geq \exp(q\theta(u))=f(u)$.  So Lemmas \ref{lem:lyap} and \ref{lem:gronwall} apply, which provide us the claim about the moments. 

We just need to show the ergodicity part. According to the arguments in \cite{MSH02}, we only need to construct a Lyapunov function for $(X_t, u_t)$.  Note that by Lipschitz condition and Lemma \ref{lem:linear}, both $b$  and $\theta$ have at most linear growth. Therefore in Lemma \ref{lem:lyap}, for a certain constant $C_3$
\[
(1+|b(u)|) f(u)\leq  C_3 \exp(C_3 \|u\| ).
\]
Recall that $\langle h(u), u\rangle \leq -\lambda \|u\|^2+M_\lambda$ for some $\lambda,M_\lambda>0$.

 Let $\CE(u)=\exp(\alpha \|u\|^2)$, with an $\alpha<\lambda$. Apply the generator, 
\begin{align*}
\mathcal{L}\CE=2\alpha  \langle h(u), u\rangle \CE + (\alpha d_u+2\alpha^2\|u\|^2)\CE\leq (-2\epsilon \|u\|^2   + d_u+2M_\lambda)\alpha\CE.
\end{align*}
Here $\epsilon=\lambda- \alpha>0$. 
 When $\epsilon\|u\|^2 \leq 2d_u+4M_\lambda$, $\CE(x)\leq \exp(2\alpha (d_u+2M_\lambda)/\epsilon)$, otherwise $\mathcal{L}\CE_t\leq -\epsilon  \CE_t$. Therefore for some constant $M_4$,
\[
\mathcal{L}\CE\leq -\epsilon  \CE+ \alpha (d_u+2M_\lambda)  \exp(2\alpha (d_u+2M_\lambda)/\epsilon)=:-\epsilon  \CE +M_4. 
\]
We can find another constant $M_5$, so that
\[
(1+|b(u)|) f(u)\leq  C_3\exp(C_3 \|u\| )\leq \frac12\epsilon  M_5\exp(\alpha \|u\|^2).
\]
Let $\widetilde{V}(x,u)=f(u) \CE_q(x)+ M_5 \CE(u)$, then by Lemma \ref{lem:lyap}, 
\begin{align*}
\mathcal{L}\widetilde{V}(x,u)&\leq -\rho f(u) \CE_q(x)+C_\delta(1+|b(u)|) f(u)-M_5\epsilon  \CE +M_4M_5\\
&\leq -\min\{\rho,\tfrac12 \epsilon   \}\widetilde{V}(x,u)+M_4M_5. 
\end{align*}
This qualifies $\widetilde{V}$ as a Lyapunov function for the process $(X_t,u_t)$. 
\end{proof}

\subsection{Exponential tails}
\label{sec:appendexp}

\begin{proof}[Proof of Proposition \ref{prop:upperbound}]
Given a vector $\alpha\in \reals^{d_X}$, let $\CE_t=\exp\langle \alpha, X_t\rangle$,  apply the chain rule \eqref{eqn:chain}, we find
\[
\mathcal{L}\CE_t=-\langle\alpha, X_t\rangle b(u_t)  \CE_t + \frac{1}{2} \CE_t\|\sigma_x \alpha\|^2. 
\]
Consider the function $H(x)=\exp(x)(\|\alpha\|-x)$. Its derivative is $\dot{H}(x)= \exp(x)(\|\alpha\|-x-1)$, so
$H(x)$ reaches its maximum $\exp(\|\alpha\|-1)$ at $x=\|\alpha\|-1$. Therefore 
\[
H(\langle\alpha, X_t\rangle)=\CE_t (\|\alpha\|- \langle\alpha, X_t\rangle)\leq \exp(\|\alpha\|-1),
\]
and by $b(u_t)\geq 0$
\[
- b(u_t) \langle\alpha, X_t\rangle \CE_t\leq \exp(\|\alpha\|-1)b(u_t)-\|\alpha\| \CE_t b(u_t).
\]
So
\[
\mathcal{L}\CE_t\leq (\tfrac12\|\sigma_x \alpha\|^2-\|\alpha\| b(u_t)) \CE_t+\exp(\|\alpha\|-1) b(u_t).
\]
Next, we apply the generator to $f=\exp(\|\alpha\|\theta)$, by the chain rule \eqref{eqn:chain} is,
\[
\mathcal{L}f=(\|\alpha\|(b-\langle \pi, b\rangle)+\tfrac12\|\alpha\|^2 \|\nabla_u \theta\|^2)f.
\]
We apply Lemma \ref{lem:lyap} first part to  $\CE_t$ and $f=\exp(\|\alpha\|\theta)$,
\[
\mathcal{L} \CE_t  f(u_t)\leq  (\tfrac12\|\alpha\|^2 \|\nabla_u \theta\|^2+\tfrac12 \sigma_x^2\| \alpha\|^2-\|\alpha\| \langle \pi, b\rangle) \CE_t  f(u_t)
+\exp(\|\alpha\|-1) b(u_t) f(u_t).
\]
Since $\langle \pi, b\rangle>0$, and $\|\nabla_u \theta\|$ is bounded, so if we choose $\alpha$ with sufficiently small norm such that
\[
\rho_\alpha=-\left(\tfrac12\|\alpha\|^2 \sup_u\|\nabla_u \theta\|^2+\tfrac12 \sigma_x^2\| \alpha\|^2-\|\alpha\| \langle \pi, b\rangle\right)>0. 
\]
Then by Gronwall's inequality, assuming $\|\alpha\|\leq 1$, 
\[
\E \CE_t  \exp(\|\alpha\| \theta(u_t))\leq e^{-\rho_\alpha t}\E \CE_0  \exp(\|\alpha\| \theta(u_0))+\int^t_0 e^{-\rho_\alpha (t-s)}\E b(u_s)f(u_s)\rmd s<\infty.
\]
Finally by H\"{o}lder's inequality, for any $\rho<1$
\[
\limsup_{t\to\infty}\E \exp(\rho \langle\alpha, X_t\rangle) \leq \limsup_{t\to\infty} [\E \CE_t \exp(\|\alpha\| \theta(u_t))]^\rho  [\E  \exp(-\rho (1-\rho)^{-1}\|\alpha\| \theta(u_t))]^{1-\rho}<\infty. 
\]
To get our claim in the proposition, one simply lets $\alpha=\rho^{-1}\beta$, with a proper $\rho<1$ so that $\rho_\alpha>0$. 

For the last claim,  note that if $\|\cdot\|_\infty$ denotes the $l_\infty$ norm, then for any positive $a$, 
\[
\E \exp(a\|X_t\|)\leq \E \exp(a\sqrt{d_X}\|X_t\|_\infty)\leq \sum_{i=1}^{d_X}\E\exp(a\sqrt{d_X} \langle e_i, X_t\rangle) +\exp(-a\sqrt{d_X} \langle e_i, X_t\rangle). 
\]
Here $e_i$ is the $i$-th standard Euclidean basis vector, so $\langle e_i, X_t\rangle$ is the $i$-th component of $X_t$. So for sufficiently small $a$, $\limsup_{t\to\infty} \E \exp(a\|X_t\|)<\infty$. 

Finally note that by Taylor expansion of $\exp(a \|x\|)$, 
\[
\|x\|^{2p}\leq (2p)! a^{-2p} \exp(a\|x\|).
\] 
So we have our claim since
\[
\log \E \|X_t\|^{2p}\leq  \sum_{k=1}^{2p}\log k -2p\log a+\log \E \exp(a\|X_t\|)\leq 2p\log p+O(p). 
\]
\end{proof}

\begin{proof}[Proof of Lemma \ref{lem:weakdampgen}]
Consider the the following process
\[
U_{p,t}=\exp\left(\int^t_0 (-p^{\frac{1}{2^m}}\mathcal{L}g_m(u_s)-p\Gamma(g_1)(u_s))\rmd s + G_p(u_t) \right).
\]
Apply the generator to the $u_t$ part, we find that 
\begin{align*}
\mathcal{L}U_{p,t}&=(\mathcal{L} G_p+\Gamma(G_p)-p^{\frac{1}{2^m}}\mathcal{L}g_m-p\Gamma(g_1))U_{p,t}\\
&=\left(\sum_{k=1}^{m-1}p^{\frac{1}{2^k}}\left( \mathcal{L}g_k+\Gamma(g_{k+1})\right)+\sum_{k\neq j} p^{\frac{1}{2^j}+\frac{1}{2^k}}\Gamma(g_j,g_k)\right) U_{p,t}\geq 0.
\end{align*}
By Dynkin's formula, $U_{p,t}$ is a submartingale
\[
\mathbb{E} U_{p,t}\geq \mathbb{E}\exp(G_p(u_0)). 
\]
Moreover, note that
\begin{align*}
\exp\left(-p\int^t_0 b(u_s)\rmd s+G_p(u_t)\right)&\geq \exp\left(-p\int^t_0\Gamma (g_1)(u_s) \rmd s+G_p(u_t)\right)\\
&=U_{p,t}\exp\left(\int^t_0 p^{\frac{1}{2^m}}\mathcal{L}g_m(u_s) \rmd s\right)\geq \exp(-p^{\frac{1}{2^m}}M t)U_{p,t}. 
\end{align*}
And by Cauchy  Schwarz
\[
\mathbb{E} \exp\left(-2p\int^t_0 b(u_s)\rmd s\right)\mathbb{E}\exp(2G_{p}(u_t))\geq \left(\mathbb{E}\exp\left(-p\int^t_0 b(u_s)\rmd s+G_p(u_t)\right)\right)^2.
\]
As a consequence
\[
\mathbb{E} \exp\left(-2p\int^t_0 b(u_s)\rmd s\right)
\geq \frac{\exp(-2 p^{\frac{1}{2^m}}M t)  \left(\mathbb{E}\exp(G_p(u_0))\right)^2}{\mathbb{E}\exp(2G_{p}(u_t))}.
\]
Then  by  Jensen's inequality,  $\mathbb{E}\exp(G_p(u_0))\geq \exp(\mathbb{E}G_p(u_0))\geq \exp(- \sqrt{p}M_1)$, moreover
$\mathbb{E}\exp(2G_{p}(u_t))\leq \exp(2\sqrt{p}M_0)$. Therefore
\[
\mathbb{E} \exp\left(-2p\int^t_0 b(u_s)\rmd s\right)\geq \exp(-2 p^{\frac{1}{2^m}}M t-2\sqrt{p}M_0-2\sqrt{p}M_1). 
\]
\end{proof}

\begin{proof}[Proof of Proposition \ref{prop:weakinter}]
We will only look at integer $p$. For non-integer $p$, one can get similar bounds using H\"{o}lder's inequality. By the Duhamel's formula, we can write
\[
X_t=A_{0,t}X_0+ \sigma_x \int^t_{0}A_{s,t}\rmd W_s, \quad A_{s,t}:=\exp\left(-\int^t_s b(u_r)dr\right). 
\]
Conditioned on the realization of $u_s$, $A_{0,t}X_0$ and $\int^t_{0}A_{s,t}\rmd W_s$ are independent, so the conditional expectation of $\|X_t\|^{2p}$ will be larger than the conditional expectation of $\| \sigma_x \int^t_{0}A_{s,t}\rmd W_s\|^{2p}$. So without loss of generality, we can assume $X_0=0$. 

Next, conditioned on the realization of the $u_s$ process, $X_t$ has a Gaussian distribution with mean being $\mathbf{0}$ and the covariance matrix being
$\sigma_x^2\left(\int^t_{0} A^2_{s,t}\rmd s\right) I$. Therefore
\[
\mathbb{E} \|X_t\|^{2p}=\sigma_x^{2p}\mathbb{E}\|Z\|^{2p}\mathbb{E}\left(\int^t_{0} A^2_{s,t}\rmd s\right)^p,
\]
where $Z$ is $\mathcal{N}(\mathbf{0},I)$. Note that  
\[
\mathbb{E}\left(\int^t_{0} A^2_{s,t}\rmd s\right)^p
=\int_{s_1,\ldots, s_p=0}^t\mathbb{E} A^2_{s_1,t}\cdots A^2_{s_p,t} \rmd s_1\cdots \rmd s_p,
\]
and because $A^2_{r,t}=A^2_{r,s}A^2_{s,t}\leq A^2_{s,t}$ for any $r\leq s\leq t$, therefore
\[
\mathbb{E}\left(\int^t_0 A^2_{s,t}\rmd s\right)^p
\geq \int_{s_1,\ldots, s_p=0}^t\mathbb{E} A^{2p}_{s_*,t}\rmd s_1\cdots \rmd s_p,
\]
where $s_*=\min\{s_1,\ldots, s_p\}$. Applying  Lemma \ref{lem:weakdampgen} with a time shift of $s$, we find that
\[
\mathbb{E} A^{2p}_{s_*,t}\geq \exp(-2p^{\frac{1}{2^m}}M(t-s_*)-2\sqrt{p}M_0-2\sqrt{p}M_1) .
\]
Notice the volume inside $\{(s_1,\ldots, s_p): s_i\leq t\}$ corresponding to $s_*\geq s$ is $(t-s)^{p}$, so by a change of variable,
\begin{align*}
\mathbb{E}\left(\int^t_0 A^2_{s,t}\rmd s\right)^p&\geq \exp(-2\sqrt{p}M_0-2\sqrt{p}M_1)\int^t_{0}\exp(-2p^{\frac{1}{2^m}}M(t-s)) p(t-s)^{p-1} \rmd s\\
&=p\exp(-2\sqrt{p}M_0-2\sqrt{p}M_1)\int_{s=0}^t \exp(-2p^{\frac{1}{2^m}}Ms)s^{p-1}\rmd s.
\end{align*}
By Lemma \ref{lem:expmoment} in below, 
\[
\log \int_{s=0}^\infty \exp(-p^{\frac{1}{2^m}}Ms)s^{p-1}\rmd s= \left(1-\frac{1}{2^m}\right)p\log p+O(p). 
\] 
Using the inequality above, we find that
\[
\liminf_{t\to\infty}\log\mathbb{E}\|X_t\|^{2p}\geq \log \mathbb{E}\|Z\|^{2p}+ \left(1-\frac{1}{2^m}\right)p\log p+O(p)=\left(2-\frac{1}{2^m}\right)p\log p+O(p). 
\]
The $\log \mathbb{E}\|Z\|^{2p}=p\log p+O(p)$ can be obtained by standard Gaussian moment formula or using Lemma \ref{lem:expmoment}. 
\end{proof}

\begin{lem}
\label{lem:expmoment}
Fixed any $r>0$, then with any sequence $c_p=O(p)$, the following holds
\[
\log \left(\int^\infty_0 \exp(-c_px^r) x^p \rmd x \right)=\frac{p}{r}\log \frac{p}{c_p}+O(p). 
\]
Here a term is $O(p)$, if this term  is bounded by $[-Mp, Mp]$ for a constant $M$ independent of $p$. 
\end{lem}
\begin{proof}
Applying the change of variable with $y=c_px^r$,  the integral can be written as
\[
\int^\infty_0 \exp(-c_px^r) x^p \rmd x=\frac{1}{r}c_p^{-\frac{p+1}{r}}\int^\infty_0\exp(-y) y^{\frac{p+1}{r}-1}\rmd y. 
\]
Let $q=\frac{p+1}{r}-1$, and denote
\[
M_q:=\int^\infty_0 \exp(-y)y^q \rmd y.
\]
Using integration by parts, we find that 
\[
M_q=q\int^\infty_0 \exp(-y)y^{q-1} \rmd y=q M_{q-1}=\cdots=\left(\prod_{k=0}^{\lfloor q\rfloor-1} (q-k)\right) M_{q-\lfloor q\rfloor}.
\]
Since $q-\lfloor q\rfloor\in[0,1)$, $M_{q-\lfloor q\rfloor}$ is bounded by constants from both below and above. Moreover, note that
\[
\log \left(\prod_{k=0}^{\lfloor q\rfloor-1} (q-k)\right)=\sum_{k=0}^{\lfloor q\rfloor-1} \log (q-k).
\]
For $u,z\in [n,n+1]$, $\frac12u\leq z\leq 2u$, so $-\log 2+\log u\leq \log z\leq \log2 +\log u$, 
\begin{equation}
\label{tmp:logsum}
-\log2+\int^{n+1}_n \log u   \rmd u\leq \log z \leq \log2+ \int^{n+1}_n \log u  \rmd u.
\end{equation}
Combining these inequalities, and returning to the formulation of $M_q$, we can conclude that
\[
\log M_q = \int^q_1 \log u+O(q)=q\log q+O(q)=\frac{p}{r}\log p+O(p). 
\]
Then our claim holds as long as  
\[
\log \left(c_p^{-\frac{p+1}{r}}\right)=-\frac{p}{r}\log c_p+O(p).
\]
Our assumption on $c_p$ guarantees this.
\end{proof}

\begin{proof}[Proof of Lemma \ref{lem:dissi}]
Without loss of generality, we will assume $u_*$ is the origin. 
We will use 
\[
g_1(u)=-M_1 \|u\|^{2^{m}}, \quad M_1:= \frac{\sqrt{C}}{2^m}. 
\]
It is easy to see that 
\begin{equation}
\label{tmp:g1}
\nabla_u g_1=-2^m M_1 \|u\|^{2^m-2} u,\quad \nabla^2_u g_1= -2^m M_1 \|u\|^{2^m-2} I_d- 2^m (2^m-2)M_1 \|u\|^{2^m-4} uu^T. 
\end{equation}
The level-1 constraint is met because
\[
\Gamma(g_1)=2^{2m} M_1^2 \|u\|^{2^{m+1}-2}\geq b(u). 
\]
Next, we notice that,
\begin{align}
\notag
\mathcal{L} g_1&=-2^m M_1 \|u\|^{2^m-2} \langle u, h(u)\rangle +\frac12\text{tr}( \nabla_u^2 g_1)\\
\notag
&=-2^m M_1 \|u\|^{2^m-2} \langle u, h(u)\rangle-\frac12 2^m(2^m-2+d_u) M_1 \|u\|^{2^m-2}\\
\label{tmp:Lg1}
&\geq \lambda 2^m M_1 \|u\|^{2^m }-2^{m-1} (2^m+d_u+2M_\lambda) M_1 \|u\|^{2^m-2}. 
\end{align}
So if we let 
\[
C_1=2^{m-1}(2^m+2M_\lambda) M_1,\quad g_2(u)=-M_2 \|u\|^{2^{m-1}}, \quad M_2:= \sqrt{\frac{C_1}{2^{m}}},
\]
 the gradient and Hessian of $g_2$ are similar to \eqref{tmp:g1}. In particular, it solves the level-2 constraint with $g_1$ since,
\[
\Gamma(g_2)\geq  2^m M_2^2  \|u\|^{2^{m}-2}=C_1 \|u\|^{2^m-2}. 
\]
And a similar lower bound as \eqref{tmp:Lg1} holds for $\mathcal{L} g_2$ with a $C_2>0$:
\[
\mathcal{L} g_2\geq \lambda 2^{m-1} M_1 \|u\|^{2^{m-1} }-C_2 \|u\|^{2^{m-1}-2}. 
\]
Clearly we can iterate this construction, and obtain a series of $g_k(u)=-M_k \|u\|^{2^{m+1-k}}, k=1,\cdots,m$, while
\[
\mathcal{L}g_{k-1}+\Gamma(g_k)\geq 0.
\]
On the other hand, we can verify that for $g_m(u)=-M_m\|u\|^2$,
\[
\mathcal{L} g_m=-2 M_m \langle u, h(u)\rangle-M_m d_u \geq 2\lambda M_m \|u\|^2 -2M_m M_\lambda -M_m d_u\geq  -2M_m M_\lambda -M_m d_u.\\
\]
Finally, we check conditions (5) and (6) in Definition \ref{defn:Am}. The first part of (5) holds as $G_p\leq 0$, and the second part holds since the power of $p$ in $G_p$ is at most $\frac12$. 

As for the alignment condition (6),  note that 
\[
\nabla_u g_k=-2^{m+1-k} M_k \|u\|^{2^{m-1-k}} u,
\]
so $\Gamma(g_j,g_k)\geq 0$.
\end{proof}

\begin{proof}[Proof of Theorem \ref{thm:exponential}]
We will first consider a process $Y_t$ defined by 
\[
\rmd Y_t=-b'(u_t)Y_t\rmd t+\sigma_x \rmd W_t,\quad Y_0=X_0.
\]
By Proposition \ref{prop:compare}, we know $\E \|Y_t\|^{2p}\geq \E \|X_t\|^{2p}$. So it suffices to prove the same upper bound for $\E \|Y_t\|^{2p}$.
Since $b'$ is Lipschitz, by Lemma \ref{lem:cauchy}, we know that $\theta$ with $\btilde=b'-q^{-1}\delta |b'|$ is well defined and Lipschitz, so $\Gamma(\theta)=\frac12\|\nabla_u \theta\|^2$ is bounded. 
This indicates that $\exp(C\theta(u))\leq \exp(CM\|u\|+CM)$ for a constant $M$. Then by Theorem \ref{thm:Gaussiantail}, and the fact the energy is dissipative, we have $\E \exp(CM\|u_t\|+CM)$ is bounded uniformly in time. Therefore, we can use Proposition \ref{prop:upperbound} to find the upper tail.

For the other direction, Corollary \ref{cor:dissi} indicates that $(b,h)\in \mathcal{A}_m$ for any $m$. So Proposition \ref{prop:weakinter} indicates the lower bound. 
\end{proof}

\subsection{Verification of the univariate case}
\label{sec:appenduni}
\begin{proof}[Proof of Theorem \ref{thm:simple}]
First of all, Theorem \ref{thm:unstable}-3) verifies that \eqref{sys:1dim} is geometrically ergodic. Next we verify the claims on the tail of $|X_t|$, considering different scenarios.

In the scenario of i), if $b$ takes negative value at $u_*$, $\E b(u_t)>0$,  and it is Lipschitz, the conditions of  Theorem \ref{thm:unstable}-1) and 2)  are met. So 
\[
p_0:=\inf\{p>0: \lim_{t\to \infty} \E \|X_t\|^p=\infty\}
\]
is well defined and strictly positive. This shows $|X_t|$ is polynomial like. 

In the scenario of ii), Lemma \ref{lem:OUiscontractive} shows that $u_t$ is asymptotically contractive. Also if we pick $u_*$ the midpoint of the interval where $b$ takes value $0$, the conditions of Theorem \ref{thm:exponential} are satisfied, therefore $|X_t|$ has an exponential-like tail. 

The scenario of iii) is similar to ii), except that there is not an interval for $b$ to take value $0$. But the conditions of Proposition \ref{prop:upperbound} still hold by Theorem \ref{thm:Gaussiantail}, and the conditions of Proposition \ref{prop:weakinter} still hold by Corollary \ref{cor:dissi} with $n=1$ and the Lipschitz condition of $b$. 

The scenario of iv) satisfies the conditions listed by Theorem \ref{thm:Gaussiantail}. 
\end{proof}
\subsection{General conditional Gaussian models}
\begin{proof}[Proof of Lemma \ref{lem:EqGauss}]
Simply note that 
\begin{align*}
\langle B(u) x, x \rangle &= \frac12\langle(B(u)+B(u)^T) x, x \rangle\\
&=\frac12\sum_{i=1}^n\langle b_i(u)(B_i+B_i^T) x, x \rangle\\
&\leq \frac12\sum_{i=1}^n  b_i M_i \vee b_i m_i  \langle I_{N_i} x,x  \rangle\\
&\leq \sum_{i=1}^n  b_i M_i \vee b_i m_i \sum_{j\in N_i} x^2_j \leq \underline{b}(u)\|x\|^2. 
\end{align*}
Use the derivatives derived in Lemma \ref{lem:Eq} for $\CE_q$, apply Young's inequality, we find that  for any fixed $\delta>0$, there is a $C_\delta$
\begin{align*}
\mathcal{L} \CE_q (x) &=  - \left( \frac{q \|x\|^q}{1+\|x\|^2}+\frac{2\|x\|^q}{(1+\|x\|^2)^2}\right)\langle B(u) x, x \rangle +\frac12\text{tr}(\Sigma_X\nabla^2 \CE_q(x)\Sigma_X^T)\\
&\geq -q\underline{b}(u)\CE_q(x)-\underline{b}(u)O((1+\|x\|)^{q-2})-O((1+\|x\|)^{q-1})\\
&\geq -(q\underline{b}+\delta|\underline{b}|+\delta)\CE_q(x)-C_\delta (1+|\underline{b}|).
\end{align*}
The converse direction comes in very similarly.
\end{proof}

\bibliographystyle{plain}
\bibliography{ref}

\begin{thebibliography}{10}

\bibitem{BH12}
Y.~Bakhtin and T.~Hurth.
\newblock Invariant densities for dynamical systems with random switching.
\newblock {\em Nonlinearity}, 25(10):2937--2952, 2012.

\bibitem{BHM15}
Y.~Bakhtin, T.~Hurth, and J.~C. Mattingly.
\newblock Regularity of invariant densities for 1{D} systems with random
  switching.
\newblock {\em Nonlinearity}, 28(11):3755--3787, 2015.

\bibitem{BGL13}
D.~Bakry, I.~Gentil, and M.~Ledoux.
\newblock {\em Analysis and geometry of Markov diffusion operators}.
\newblock Grundlehren der mathematischen Wissenschaften. Springer, 2013.

\bibitem{BGM10}
J.~B. Bardet, H.~Gu\'{e}rin, and F.~Malrieu.
\newblock Long time behavior of diffusions with {M}arkov switching.
\newblock {\em ALEA Lat. Am. J. Probab. Math. Stat.}, 7:151-- 170, 2010.

\bibitem{BGM12}
M.~Branicki, B.~Gershgorin, and A.~J. Majda.
\newblock Filtering skill for turbulent signals for a suite of nonlinear and
  linear extended {K}alman filters.
\newblock {\em J. Comput. Phys.}, 231(4):1462--1498, 2012.

\bibitem{Cetal89}
B.~Castaing, G~Guanratne, F~Heslot, L~Kadanoff, A~Libchaber, S~Thomae, X~Wu,
  S~Zaleski, and G~Zanetti.
\newblock Scaling of hard thermal turbulence in rayleigh-b\'{e}nard convection.
\newblock {\em J. Fluid Mech.}, 204(1):1--30, 1989.

\bibitem{CD15}
L.~Chen and R.~C. Dalang.
\newblock Moments and growth indices for the nonlinear stochastic heat equation
  with rough conditions.
\newblock {\em Annals of Probability}, 2015.

\bibitem{CKK17}
L.~Chen, D.~Khoshnevisan, and K.~Kim.
\newblock A boundedness trichotmoy for the stochastic heat equation.
\newblock {\em Ann. Inst. H. Poincar{\'e} Probab. Statist.}, 53(4):1991--2004,
  2017.

\bibitem{CGHM14}
N.~Chen, D.~Giannakis, R.~Herbei, and A.~J. Majda.
\newblock An {MCMC} algorithm for parameter estimation in signals with hidden
  intermittent instability.
\newblock {\em SIAM/ASA J. Uncertainty Quantification}, 2(1):647--669, 2014.

\bibitem{CM18}
N.~Chen and A.~J. Majda.
\newblock Efficient statistically accurate algorithms for the
  {F}okker--{P}lanck equation in large dimensions.
\newblock {\em J. Comput. Phys.}, 354:242--268, 2018.

\bibitem{CMG14}
N.~Chen, A.~J. Majda, and D.~Giannakis.
\newblock Predicting the cloud patterns of the {M}adden‐{J}ulian oscillation
  through a low‐order nonlinear stochastic model.
\newblock {\em Geophysical Res. Lett.}, 41(15):5612--5619, 2014.

\bibitem{CMT17}
N.~Chen, A.~J. Majda, and X.~T. Tong.
\newblock Rigorous analysis for efficient statistically accurate algorithms for
  solving fokker-planck equations in large dimensions.
\newblock arXiv:1709.05585.

\bibitem{CMT14}
N.~Chen, A.~J. Majda, and X.~T. Tong.
\newblock Information barriers for noisy lagrangian tracers in filtering random
  incompressible flows.
\newblock {\em Nonlinearity}, 27:2133--2163, 2014.

\bibitem{CMT14b}
N.~Chen, A.~J. Majda, and X.~T. Tong.
\newblock Noisy lagrangian tracers for filtering random rotating compressible
  flows.
\newblock {\em J. Non. Sci.}, 25(3):451--488, 2014.

\bibitem{CH14}
B.~Cloez and M.~Hairer.
\newblock Exponential ergodicity for {M}arkov processes with random switching.
\newblock {\em Bernoulli}, 21(1):505--536, 2015.

\bibitem{Easterling_etal}
D.~R. Easterling, J.~L. Evans, P.~Y. Groisman, T.~R. Karl, K.~E. Kunkel, and
  P.~Ambenje.
\newblock Observed variability and trends in extreme climate events: a brief
  review.
\newblock {\em Bulletin of the American Meteorological Society},
  81(3):417--425, 2000.

\bibitem{Ebe16}
A.~Eberle.
\newblock Reflection couplings and contraction rates for diffusions.
\newblock {\em Probab. Theory Related Fields}, 166(3-4):851--886, 2016.

\bibitem{EGZ16}
A.~Eberle, A.~Guillin, and R.~Zimmer.
\newblock Quantitative {H}arris type theorems for diffusions and
  {M}c{K}ean-{V}lasov.
\newblock arXiv:1606.06012.

\bibitem{EMN12}
B.~Eichengreen, A.~Mody, M.~Nedeljkovic, and L.~Sarno.
\newblock How the subprime crisis went global: evidence from bank credit
  default swap spreads.
\newblock {\em ournal of International Money and Finance}, 31(5):1299--1318,
  2012.

\bibitem{Neelin_etal}
J.~Neelin et al.
\newblock Long tails in deep columns of natural and anthropogenic tropospheric
  tracers.
\newblock {\em Geophysical Res. Lett.}, 37:L05804, 2010.

\bibitem{Gollub_etal}
J.~P.~Gollub et al.
\newblock Fluctuations and transport in a stirred fluid with a mean gradient.
\newblock {\em Phys. Rev. Lett.}, 67:3507--3510, 1991.

\bibitem{FS17}
M.~Farazmand and T.~Sapsis.
\newblock Reduced-order prediction of rogue waves in two dimensional water
  waves.
\newblock {\em J. Comput. Phys.}, 340:418--434, 2017.

\bibitem{GHM10}
B.~Gershgorin, J.~Harlim, and A.~J. Majda.
\newblock Test models for improving filtering with model errors through
  stochastic parameter estimation.
\newblock {\em J. Comput. Phys.}, 229:1--31, 2010.

\bibitem{GM11}
B.~Gershgorin and A.~J. Majda.
\newblock Filtering a statistically exactly solvable test model for turbulent
  tracers from partial observations.
\newblock {\em J. Comput. Phys}, 230:1602--1638, 2011.

\bibitem{GX17}
Y.~Gu and W.~Xu.
\newblock Moments of 2d parabolic {A}nderson model.
\newblock https://arxiv.org/abs/1702.07026.

\bibitem{HMS06}
D.~J. Higham, X.~Mao, and A.~M. Stuart.
\newblock Strong convergence of {E}uler-type methods for nonlinear stochastic
  differential equations.
\newblock {\em SIAM J. Numerical Analysis}, 40(3):1041--1063, 2006.

\bibitem{HD14}
R.~Huser and A.~C. Davison.
\newblock Space-time modelling of extreme events.
\newblock {\em J. Royal Statistical Society: Series B}, 76(2):439--461, 2014.

\bibitem{KBM10}
B.~Khouider, J.~A. Biello, and A.~J. Majda.
\newblock A stochastic multicloud model for tropical convection.
\newblock {\em Comm. Math. Sci.}, 8(1):187--216, 2010.

\bibitem{KMS13}
B.~Khouider, A.~J. Majda, and S.~Stechmann.
\newblock Climate science in the tropics: Waves, vortices, and pdes.
\newblock {\em Nonlinearity}, 26(R1-R68), 2013.

\bibitem{LMQ16}
Y.~Lee, A.~J. Majda, and D.~Qi.
\newblock Stochastic superparameterization and multiscale filtering of
  turbulent tracers.
\newblock {\em SIAM Multiscale Model. Simul.}, 14(1), 2016.

\bibitem{LS01}
R.~S. Liptser and A.~N. Shiryaev.
\newblock {\em Statistics of random processes. I, II}, volume~5 of {\em
  Applications of Mathematics}.
\newblock Springer-Verlag, 2001.

\bibitem{MFK08}
A.~J. Majda, C.~Franzke, and B.~Khouider.
\newblock An applied mathematics perspective on stochastic modelling for
  climate.
\newblock {\em Phil. Trans. Roy. Soc.}, 336(1875):2427--2453, 2008.

\bibitem{MG13}
A.~J. Majda and B.~Gershgorin.
\newblock Elementary models for turbulent diffusion with complex physical
  features: eddy diffusivity, spectrum, and intermittency.
\newblock {\em Phil. Trans. Roy. Soc.}, 371(1982), 2013.

\bibitem{MH12}
A.~J. Majda and J.~Harlim.
\newblock {\em Filtering complex turbulent systems}.
\newblock Cambridge University Press, Cambridge, UK, 2012.

\bibitem{Majda_Harlim_Gershgorin}
A.~J. Majda, J.~Harlim, and B.~Gershgorin.
\newblock Mathematical strategies for filtering turbulent dynamical systems.
\newblock {\em Discrete and Continuous Dynamical Systems}, 27:441--486, 2010.

\bibitem{MS09}
A.~J. Majda and S.~Stechmann.
\newblock The skeleton of tropical intraseasonal oscillations.
\newblock {\em Proc. Natl. Acad. Sci.}, 106(21):8417--8422, 2009.

\bibitem{MS09b}
A.~J. Majda and S.~Stechmann.
\newblock The skeleton of tropical intraseasonal oscillations.
\newblock {\em Proc. Natl. Acad. Sci.}, 106:8417--8422, 2009.

\bibitem{MT15}
A.~J. Majda and X.~T. Tong.
\newblock Intermittency in turbulent diffusion models with a mean gradient.
\newblock {\em Nonlinearity}, 28(11), 2015.

\bibitem{MT16RMS}
A.~J. Majda and X.~T. Tong.
\newblock Moment bounds and geometric ergodicity of diffusions with random
  switching and unbounded transition rates.
\newblock {\em Research in the mathematical sciences}, 3(41), 2016.

\bibitem{Majda_Wang}
A.~J. Majda and X.~Wang.
\newblock {\em Nonlinear Dynamics and Statistical Theories for Basic
  Geophysical Flows}.
\newblock Cambridge University Press, Cambridge, UK, 2006.

\bibitem{MSH02}
J.~C. Mattingly, A.~M. Stuart, and D.~J. Higham.
\newblock Ergodicity for {SDE}s and approximations: Locally {L}ipschitz vector
  fields and degenerate noise.
\newblock {\em Stochastic processes and their applications}, 101(2):185--232,
  2002.

\bibitem{Sha15}
J.~Shao.
\newblock Strong solutions and strong feller properties for regime-switching
  diffusion processes in an infinite state space.
\newblock {\em SIAM J. Control Optim}, 53(4):2462--2479, 2015.

\bibitem{tsay05}
R.~S. Tsay.
\newblock {\em Analysis of financial time series}.
\newblock John Wiley \& Sons, 2005.

\end{thebibliography}

\end{document}